\newif\ifms
\msfalse
\ifms
    \documentclass[11pt]{article} 
    \usepackage[margin=1in]{geometry}
\else
    \documentclass[11pt, letterpaper]{article}
    \usepackage[margin=1in]{geometry}
\fi

  \usepackage[centertags]{amsmath}
  \usepackage{amsfonts}
  \usepackage{amsthm}
  \usepackage{amssymb}
  \usepackage{bm}
  \usepackage{latexsym}
  \usepackage{graphics}
  \usepackage{graphicx}
  \usepackage{subcaption} 
  \usepackage{tabularx}
  \usepackage{caption}

  \usepackage[hidelinks]{hyperref}
  \ifms
  \else
  \hypersetup{
    pdftitle={Supply Chain Networks: Equilibrium, Sensitivity, Dynamics and Shortage Games},
    pdfauthor={Elioth Sanabria},
    pdfsubject={Research article},
    pdfkeywords={Supply chain networks; network clearing; Newsvendor; M-matrices; disruption cascades; comparative statics},
    pdfstartview={FitH}
}
\fi
\usepackage{algorithm} 
\usepackage{algorithmicx}
\usepackage{algpseudocode}
\usepackage[style=authoryear, maxcitenames=2, maxbibnames=99]{biblatex}
\renewbibmacro{in:}{%
  \ifentrytype{article}{}{\printtext{\bibstring{in}\intitlepunct}}%
}
\renewbibmacro*{volume+number+eid}{%
  \printfield{volume}%
  \setunit*{\addnbspace}%
  \printfield{number}%
  \setunit{\addcomma\space}%
  \printfield{eid}}
\DeclareFieldFormat[article]{number}{\mkbibparens{#1}}
\usepackage{tabularx}
\usepackage{array} 
\usepackage{booktabs}
\newcolumntype{L}[1]{>{\raggedright\arraybackslash}p{#1}}
\newcolumntype{R}[1]{>{\raggedleft\arraybackslash}p{#1}}
\newcolumntype{C}[1]{>{\centering\arraybackslash}p{#1}}

\usepackage{tikz}
\usetikzlibrary{graphs, positioning, calc, arrows.meta, 3d, fit, backgrounds, matrix, decorations.pathreplacing}

\usepackage{pgfplots}
\pgfplotsset{compat=1.17}
\usepgfplotslibrary{fillbetween, groupplots}
\usepackage{pgfplotstable}

\usepackage{fancyhdr}
\ifms
\else
    
\fi

\fancypagestyle{plain}{
    \fancyhf{}
    \fancyfoot[C]{\small\thepage}
    
}

\newcommand{\pr}{\mbox{\sf P}}
\newcommand{\ex}{{\bf\sf E}}               
\newcommand{\var}{\mbox{\sf Var}}

\newcommand{\cals}{{\cal S}}

\newcommand{\cala}{{\cal A}}

\newcommand{\calh}{{\cal H}}

\def\ind{{\bf 1}}

\newcommand\norm[1]{\left\lVert#1\right\rVert}


\theoremstyle{definition}
\newtheorem{thm}{Theorem}
\newtheorem{lem}[thm]{Lemma}
\newtheorem{pro}[thm]{Proposition}

\newtheorem{defn}{Definition}
\newtheorem{asm}{Assumption}
\newtheorem{rem}{Remark}

\ifms
    \newtheorem{exm}{Example} 
\else
    \newtheorem{exm}{Example}[section] 
\fi

\title{\vspace{-1.5cm} Supply Chain Networks: Equilibrium, Sensitivity, Dynamics and Shortage Games}

\ifms
    \author{Anonymous Authors}
    \date{\small \today}
\else
    \author{
        Elioth Sanabria \\[0.2cm]
        \small Department of Decision and Technology Analytics \\ 
        \small Lehigh University - College of Business\\
        \small \texttt{els626@lehigh.edu}
    }
    \date{\small \today}
\fi

\usepackage{placeins}
\begin{document}

\maketitle
\begin{abstract}
\noindent We model production in a supply chain network in which every node sets a Newsvendor order target, upstream demand is the sum of downstream orders routed through a substochastic matrix, and every node has a capacity. The clearing vector of production quantities is the unique Nash equilibrium of the ordering game and the unique optimal solution of a linear program, and it has the form of the Eisenberg--Noe interbank clearing model. Nodes that produce their order target are hedged; the others are unhedged. We give comparative statics and a closed form for the stationary covariance of node demands. Inventory drawdown after a capacity shock is a Skorokhod reflection problem with a closed-form solution: the network and the unhedged nodes determine who is eventually rationed and by how much, and inventories determine when. When rationed nodes can inflate their orders, the equilibria never lower aggregate deliveries and can make suppliers with no physical shortage unhedged. With endogenous prices, a per-epoch price equilibrium exists and, under mild conditions, is unique when supply is price-elastic. We illustrate the model on global crude-oil trade.
\vspace{0.4cm} 
\end{abstract}

\vspace{0.5cm}
\setlength{\parskip}{0.4em}
\section{Introduction}
In a supply network a firm sources from several suppliers and sells to firms that do the same, so the demand it faces depends on the orders and the capacities of its customers.

A profit-maximizing firm facing stochastic demand balances the marginal cost of over-ordering against the penalty of under-ordering. In single-firm or decoupled multi-echelon settings, this trade-off is resolved by the classical Newsvendor model, which sets an inventory buffer at expected demand plus a safety stock proportional to demand volatility. In a network, the demand an upstream producer experiences is an endogenous function of the ordering decisions, safety-stock choices, and capacity constraints of its downstream network.

The quantity an upstream node produces is determined jointly by the order targets of every downstream node and by the capacities of every node. Two questions arise. First, is there a well-defined, unique production profile consistent with every node optimizing against the orders it receives, and how is it computed? Second, how does that profile respond to a change in demand, volatility, or capacity at one node, and in particular how far upstream and downstream does a localized capacity loss propagate? We answer both by casting the network as a clearing problem.

Prices and safety stocks also move together. A capacity disruption restricts total physical throughput, market clearing raises the price through the elasticity of demand and supply, the higher price shifts the critical ratios and with them the safety-stock targets, and nodes rebalance their orders away from depleted or bottlenecked corridors. The coupling of physical depletion rates with price feedback determines how fast inventories are drawn down and which nodes become unhedged.

We propose a tractable model by combining ingredients found in separate literatures, that result in a general modeling paradigm to study supply chain networks. Fluid network models have capacities and reflection, but they take the input rates as given \parencite{chen2001fundamentals}. Multi-echelon inventory models derive the order targets from the critical ratio, but they require an acyclic network \parencite{clark1960optimal,graves2000optimizing}. Production-network models have the multiplier $(I-P)^{-1}$ on general topologies, with no capacities and no inventories \parencite{acemoglu2012network}. Interbank clearing models have caps and cascades on general topologies \parencite{eisenberg2001systemic}, and \textcite{chen2021financial} add liquid assets and default times to them. Finally, order inflation under rationing has been analyzed for a single supplier \parencite{cachon1999capacity}. In this paper we put endogenous orders, capacities and inventories on the same network. Theorem~\ref{thm:static_solution} adapts the result of \textcite{eisenberg2001systemic} under the correspondence of Section~\ref{sec:model}, and the drawdown model of Section~\ref{sec:dynamic} builds on the dynamic clearing model of \textcite{chen2021financial}. In our setting the requirements are generated by the order targets, and a hedged node has zero drift. As a result, a node is eventually rationed if and only if it is unhedged or lies downstream of an unhedged node (Proposition~\ref{prop:exposure}(i)). The main new results are for the case in which nodes can inflate their orders. An inflated order reaches every supplier of the inflating node, so a supplier with spare capacity can become unhedged and ration customers that have no link to the shock (Theorem~\ref{thm:network}, Example~\ref{exm:bottleneck}).

We track the propagation of mean demand and of demand covariance across the network, and we model price-sensitive safety-stock parameters.
The result is a network clearing model in the sense of \textcite{eisenberg2001systemic}, applied to operations: capacities play the role of nominal obligations, order targets the role of operating cash flows, and the set of hedged nodes the role of the default set. It supplies a linear-programming characterization of the equilibrium, exact multipliers for shock propagation, a Skorokhod-problem model of inventory drawdown with explicit times to exhaustion, a shortage game for order inflation under rationing, and an endogenous-price extension.

\subsection{Contributions of the Paper}
\begin{itemize}
    \item \textbf{A capacitated clearing model for general networks.} Each node's production is the minimum of its Newsvendor target, which aggregates its own external demand and the orders of its customers, and its capacity. As in \textcite{eisenberg2001systemic}, the clearing vector exists, is unique, is the unique Nash equilibrium of the decentralized ordering game, and is the unique optimal solution of a linear program whose dual prices are network multipliers (Theorem \ref{thm:static_solution}). The network partitions into hedged nodes $\calh$ and unhedged nodes $U$, with $q_\calh = (I_\calh - P_\calh)^{-1}(\alpha_\calh + P_{\calh,U} C_U)$.
    \item \textbf{Exact comparative statics and covariance propagation.} We give the derivatives of the clearing vector with respect to mean demand, capacity, and external volatility (Proposition \ref{prop:sens}). For the pass-through order process the stationary covariance of node demands has the closed form $\sigma^2 = N_\sigma w^2$ with $N_\sigma = \sum_{k\ge0} P^k \odot P^k$ (Theorem \ref{thm:variance}), which dominates the Hadamard-series form $(I - P \odot P)^{-1}$ unless inflows are arborescent. The diagonal of $(I_\calh - P_\calh)^{-1}$ exceeds one exactly when a node lies on a directed cycle, so order targets compound along cycles and pool along trees.
    \item \textbf{Dynamic multi-dimensional Skorokhod mapping.} We extend the static equilibrium into continuous time by formulating inventory drawdown as a multi-dimensional Skorokhod reflection problem whose reflection matrix is built from the clearing orders. The problem is well posed and its solution is piecewise linear with closed-form shortfall rates, at most $N$ epochs, and explicit times to exhaustion (Theorem \ref{thm:drawdown}). The long-run rationing level of every node is given in closed form and is independent of initial inventories: a node eventually runs out if and only if it is unhedged or lies downstream of an unhedged node, and additional reserves anywhere in the network never bring forward any node's exhaustion (Proposition \ref{prop:exposure}).
    \item \textbf{Order inflation under rationing (the shortage game).} When suppliers ration pro rata, customers gain by ordering more than they need. We define the resulting game on a general network. Equilibria exist, form a lattice, and always inflate orders relative to truthful ordering (Theorem \ref{thm:shortage}). In every equilibrium total served demand equals total own production, so order inflation never lowers aggregate deliveries: it redistributes them. The equilibrium is unique when every node serves a market of its own, and we give a closed-form condition under which a supplier with no physical shortage becomes unhedged (Theorem \ref{thm:network}), together with a closed form for the equilibrium at a single bottleneck (Proposition \ref{prop:bottleneck}). When every node serves a market of its own and no node supplies itself, the shortage equilibrium is the unique Nash equilibrium of the game in which each node accounts for the effect of its own order on all fill rates (Proposition~\ref{prop:nash}).
    \item \textbf{Tractability at scale.} The static equilibrium is the solution of one linear program. The sequential algorithm returns the exact drawdown path after at most $N+1$ passes, each of which evaluates closed-form expressions; there is no time discretization and the number of passes does not depend on the horizon. The long-run exposure of every node is given by one closed-form expression without simulating the path (Proposition \ref{prop:tractable}).
    \item \textbf{Endogenous prices.} Prices clear a constant-elasticity market against delivered supply and feed back into external demand and safety-stock targets. We prove existence of the per-epoch price fixed point, prove its uniqueness at every epoch when supply is price-elastic, and show that the safety-stock channel weakly enlarges the unhedged set (Proposition \ref{prop:price_equilibrium}).
\end{itemize}

\subsection{Literature Review}
We draw on four streams of literature: stochastic fluid networks, decentralized multi-echelon inventory control, systemic risk in network clearing models, and endogenous pricing under capacity constraints.

The study of stochastic processing networks via fluid and diffusion approximations has a long history in operations research. \textcite{chen2001fundamentals} develop the fluid limits of multi-class queueing networks and their asymptotic stability properties. These models treat nominal input flows as exogenous parameters. We endogenize the network's internal flow requirements: each requirement is the profit-maximizing order target of the node that receives it.

For the node-level behavior we use classical inventory theory. \textcite{arrow1951optimal} formalized the critical-ratio trade-off underlying stochastic inventory management. In multi-echelon settings, \textcite{clark1960optimal} proved the optimality of base-stock policies for serial supply chains. The Clark--Scarf structure extends to assembly and distribution systems \parencite{rosling1989optimal,langenhoff1990analytical,federgruen1984approximations}, and \citeauthor{graves2000optimizing} (\citeyear{graves2000optimizing}, \citeyear{graves2003supply}) place safety stocks in general acyclic networks under a guaranteed-service assumption. These models require acyclicity; our clearing formulation does not.

Our multiplier results are related to, but distinct from, the bullwhip literature. \textcite{lee1997bullwhip} identify behavioral and institutional sources of order-variance amplification; \textcite{chen2000quantifying} derive the variance ratio $\var(q)/\var(D) \ge 1 + 2\ell/p + 2\ell^2/p^2$ for a serial chain with lead time $\ell$ and $p$-period moving-average forecasts; \textcite{ouyang2006characterization} and \textcite{ouyang2010bullwhip} study amplification in multi-stage and networked systems; \textcite{osadchiy2021bullwhip} analyze the bullwhip effect in supply networks and show that network structure moderates aggregate amplification; and \textcite{cachon2007search} find that at the industry level, amplification is the exception rather than the rule. We do not analyze an order-variance ratio. Our objects are (i) the comparative statics of a deterministic clearing vector and (ii) the stationary covariance of a pass-through order process with zero lead time. In the latter the order of a node equals its demand, so the \textcite{chen2000quantifying} ratio is identically one; what varies across topology is how much external variance each node faces, which is what $N_\sigma$ measures. Order inflation under rationing is one of the classical drivers identified by \textcite{lee1997bullwhip} and is analyzed for a single supplier by \textcite{cachon1999capacity}. Our shortage game (Section \ref{sec:shortage}) concerns the allocation of a shortage across a network with fixed sourcing shares, not order variance: the new effects are that inflated orders reach every supplier of the inflating node and that requirements inflate upstream.

Our transient analysis also connects to the operations literature on supply chain disruptions. The time-to-recover/time-to-survive framework of \textcite{simchilevi2015risk} quantifies, for a given firm, how long a node outage can be sustained before performance degrades; our times-to-exhaustion $\tau_k$ answer the network-level analogue endogenously, with each node's survival horizon determined jointly by the routing topology and the clearing partition rather than by a single-firm stress scenario. In the economics of production networks, \textcite{acemoglu2012network} show that input-output linkages propagate microeconomic shocks into aggregate fluctuations through a Leontief multiplier, and \textcite{carvalho2021supply} document this propagation empirically following the Great East Japan earthquake; our model shares the $M$-matrix multiplier structure but couples it to inventory buffers, capacity constraints, and price-dependent safety-stock policies absent from that literature. Within operations management, \textcite{ang2017disruption} and \textcite{bimpikis2019supply} study optimal sourcing and network design under disruption risk in multi-tier systems, treating the network as a decision variable ex ante; our framework is complementary, characterizing how an arbitrary given topology behaves ex post as a disruption unfolds, and \textcite{osadchiy2016systematic} provide empirical evidence that supply network position loads systematic risk onto firms, consistent with the exposure structure we derive.

Methodologically, our static equilibrium has the same algebraic form as the systemic risk and default cascade models of financial economics. The clearing map is that of \textcite{eisenberg2001systemic}; see also \textcite{glasserman2016contagion}. Theorem~\ref{thm:static_solution} gives the corresponding existence, uniqueness, and linear-programming results for the supply chain network.

We add market-clearing prices and rebalancing of trade shares to this model. Classical supply chain network models (e.g., \cite{dong2005multitiered}) evaluate equilibrium prices with variational inequalities under static multi-criteria environments; they do not include dynamic, state-dependent safety-stock choices ($z_i(t)$). Financial network models with endogenous price feedback capture it in asset markets without physical inventories or capacities: fire sales in \textcite{cifuentes2005liquidity} and \textcite{caccioli2014stability}, and distress contagion in \textcite{veraart2020distress}. Our framework couples continuous-time physical depletion rates ($\theta$) with price elasticity curves ($\varepsilon,\eta$): prices ($p(t)$) respond to total physical throughput ($Q(t)$), and the price shifts the localized critical ratios ($\rho_i(t)$) and safety parameters ($z_i(t)$) through the inverse standard normal distribution. The unhedged set changes as a result.

Finally, our transient dynamic formulation relies on the mathematical framework of the Skorokhod Problem \parencite{skorokhod1961stochastic}. Reflected processes were introduced to govern the sample paths of stochastic differential equations restricted to closed boundaries, and are widely used to evaluate heavy-traffic processing limits in manufacturing systems (e.g., \cite{harrison2013brownian}). Existence, uniqueness, and Lipschitz continuity of the multidimensional reflection map are due to \textcite{harrison1981reflected}. The closest antecedent is \textcite{chen2021financial}, who formulate the dynamics of a financial clearing network as a Skorokhod problem; Section \ref{sec:dynamic} is the inventory counterpart, with unmet requirements in place of unpaid liabilities.

Section~\ref{sec:model} presents the static model and the covariance of node demands. Section~\ref{sec:dynamic} presents the drawdown dynamics. Section~\ref{sec:pricing} presents the shortage game, rebalancing, and price feedback. Section~\ref{sec:data_and_experiment} applies the model to the crude-oil trade network, and Section~\ref{sec:conclusion} concludes.

\section{Model}
\label{sec:model}
The model combines two elements. First, every node chooses its production quantity $q_i$ to maximize expected profit against the demand it faces, which yields a Newsvendor-type target \parencite{arrow1951optimal}. Second, the demand a node faces is its own external demand plus the orders of the nodes it supplies, routed through a routing matrix $P$ \parencite{chen2001fundamentals}. The production quantity is further bounded by a capacity. Concatenating these relations across nodes defines a Supply Chain Network; see Figure \ref{fig:supply_chain_network}.

\begin{figure}[htb]
    \centering
    \definecolor{bhdeepblue}{RGB}{148, 22, 12}    
    \definecolor{bhmidblue}{RGB}{209, 82, 62}     
    \definecolor{bhlightblue}{RGB}{240, 152, 122} 
    \definecolor{charcoal}{RGB}{70, 75, 80}       

\begin{tikzpicture}[
    node distance=2cm and 1.5cm,
    entity/.style={rectangle, draw=charcoal, thick, minimum width=2.2cm, minimum height=1cm, text=white},
    arrow/.style={-Stealth, thick, draw=charcoal},
    transship/.style={Stealth-Stealth, dashed, thick, draw=charcoal}
]
    \node[entity, fill=bhdeepblue] (P1) {Producer ($1$)};
    
    \node[entity, below left=1.5cm and 1.2cm of P1, fill=bhmidblue] (RS1) {Reseller ($2$)};
    \node[entity, below right=1.5cm and 1.2cm of P1, fill=bhmidblue] (RS2) {Reseller ($3$)};
    
    \node[entity, below=4cm of P1, fill=bhlightblue] (RT2) {Retailer ($5$)};
    \node[entity, left=1cm of RT2, fill=bhlightblue] (RT1) {Retailer ($4$)};
    \node[entity, right=1cm of RT2, fill=bhlightblue] (RT3) {Retailer ($6$)};
    
    \draw[arrow] (P1) -- node[above left, font=\small, text=black] {$p_{12}$} (RS1);
    \draw[arrow] (P1) -- node[above right, font=\small, text=black] {$p_{13}$} (RS2);
    \draw[arrow] (RS1) -- node[left, font=\small, pos=0.7, text=black] {$p_{24}$} (RT1);
    \draw[arrow] (RS1) -- node[right, font=\small, pos=0.4, text=black] {$p_{25}$} (RT2);
    \draw[arrow] (RS2) -- node[left, font=\small, pos=0.4, text=black] {$p_{35}$} (RT2);
    \draw[arrow] (RS2) -- node[right, font=\small, pos=0.7, text=black] {$p_{36}$} (RT3);

    \draw[transship] (RT1.east) -- node[above, font=\scriptsize, text=black] {$p_{45}, p_{54}$} (RT2.west);
    \draw[transship] (RT2.east) -- node[above, font=\scriptsize, text=black] {$p_{56}, p_{65}$} (RT3.west);
    
    \draw[dashed, ->, draw=charcoal] (RT1.south) -- ++(0,-0.8) node[below, font=\small, text=black] {$v_4$};
    \draw[dashed, ->, draw=charcoal] (RT2.south) -- ++(0,-0.8) node[below, font=\small, text=black] {$v_5$};
    \draw[dashed, ->, draw=charcoal] (RT3.south) -- ++(0,-0.8) node[below, font=\small, text=black] {$v_6$};
    
    \node[right=of P1, xshift=1.5cm, align=left, font=\small, draw=charcoal, dashed, inner sep=5pt, text=black] {
        \textbf{Model Variables:}\\
        $q_i$: Production Quantity\\
        $C_i$: Capacity\\
        $p_{ij}$: Flow Proportions\\
        $v_i$: Avg External Demand\\
        $\alpha_i$: $v_i + z_i\sigma_i$
    };
\end{tikzpicture}
\caption{Example of a Supply Chain Network.}
\label{fig:supply_chain_network}
\end{figure}

Nodes can take different roles typical of a traditional description of a Supply Chain Network such as producers, resellers and retailers. These roles can be totally flexible and nodes can satisfy the demand of other nodes as well as external demands. The remainder of this section describes the static model; the dynamic model is the subject of Section~\ref{sec:dynamic}.

\subsection{Static Model}
In this subsection we consider the static version of the model. The dynamic model of Section~\ref{sec:dynamic} is built on it and coincides with it when no node is unhedged (Theorem~\ref{thm:drawdown}(e)). We present the building blocks of the model and the intuition behind it.

Consider $N$ nodes. Node $i$ faces an external demand $D_i^{\mathrm{ext}}$ with mean $v_i \ge 0$ and standard deviation $w_i \ge 0$, and receives orders from other nodes. The routing matrix $P = [p_{ij}] \in \mathbb{R}^{N\times N}_+$ has $p_{ij}$ equal to the share of node $j$'s requirement sourced from node $i$, so that the $j$-th column sum $\varsigma_j := \sum_i p_{ij} \le 1$ is the fraction of $j$'s requirement sourced inside the network and $1 - \varsigma_j$ is produced by $j$ itself or sourced outside the model.

\begin{asm}[Open network]\label{defn:open}
$P \ge 0$, $\varsigma_j \le 1$ for all $j$, and $\rho(P) < 1$. Equivalently, from every node there is a directed path in $P^\intercal$ to some node $j$ with $\varsigma_j < 1$.
\end{asm}

\begin{asm}[Hedging]\label{defn:hedging}
For every $i$ the critical ratio $\rho_i = (p_i - c_i)/p_i$ lies in $(1/2, 1)$, so that $z_i = \Phi^{-1}(\rho_i) > 0$, and the standard deviation $\sigma_i > 0$ of the demand faced by node $i$ is given. ($\sigma$ is derived in Section \ref{sec:covariance}; the static model takes it as a parameter.)
\end{asm}

Under Assumptions \ref{defn:open}--\ref{defn:hedging}, $I - P$ is a nonsingular $M$-matrix, $(I - P)^{-1} = \sum_{k\ge0} P^k \ge 0$, and for any $\gamma \in (\rho(P), 1)$ the vector $u := (I - P/\gamma)^{-1} \ind > 0$ satisfies $Pu \le \gamma u$. Write $\norm{x}_u := \max_i |x_i|/u_i$.

Given the production vector $q$ of all other nodes, the demand faced by node $i$ has mean
\begin{equation}\label{eq:mean_demand}
  \mu_i(q) = \textstyle\sum_{j} p_{ij} q_j + v_i,
\end{equation}
and node $i$ chooses $q_i \in [0, C_i]$ to maximize $\ex[\pi_i] = p_i \ex[\min(D_i, q_i)] - c_i q_i$ with $D_i \sim \mathcal{N}(\mu_i(q), \sigma_i^2)$. This objective is concave in $q_i$ with unconstrained maximizer $\mu_i(q) + z_i \sigma_i$, so the best response is
\begin{equation}\label{eq:best_response}
  q_i = (\alpha_i + \textstyle\sum_{j} p_{ij} q_j) \wedge C_i, \qquad \alpha_i := v_i + z_i \sigma_i.
\end{equation}
A production vector $q$ is a clearing vector if \eqref{eq:best_response} holds for every $i$, i.e.\ $q = T(q)$ with $T(q) := (\alpha + Pq) \wedge C$.
At a clearing vector we call node $i$ \emph{hedged} if it produces its order target, $q_i = \alpha_i + (Pq)_i$, and \emph{unhedged} if its capacity holds it below the target, $\alpha_i + (Pq)_i > C_i$.

The structure is that of the interbank clearing model of \textcite{eisenberg2001systemic}, with the following correspondence: capacities $C$ correspond to nominal obligations, order targets $\alpha$ to operating cash flows, the routing matrix $P$ to the transpose of the relative liabilities matrix, and the clearing vector $q$ to the clearing payment vector. Under this correspondence the unhedged nodes are the banks that pay in full and the hedged nodes are the banks in default. The economic reading is reversed: a bank that pays its full obligation is solvent, whereas a node that produces at capacity fails to meet its order target. With these elements we can define a Supply Chain Network and the structure of its solution:
\begin{defn}[Supply Chain Network]\label{defn:scn}
  A supply chain network is a tuple $(P, v, w, \sigma, z, C)$ satisfying Assumptions \ref{defn:open}--\ref{defn:hedging} with $v, w, C \ge 0$. Its order target vector is $\alpha = v + z \odot \sigma$, and its clearing vector is a fixed point of $T(q) = (\alpha + Pq) \wedge C$.
\end{defn}

\begin{thm}[Clearing equilibrium]\label{thm:static_solution}
Under Assumptions \ref{defn:open}--\ref{defn:hedging}:
\begin{enumerate}
\item[(a)] $T$ is a contraction on $([0, C], \norm{\cdot}_u)$ with modulus $\gamma$; hence a unique clearing vector $q^*$ exists, and $q^* = \lim_n T^n(q^{(0)})$ for any $q^{(0)} \in [0, C]$.
\item[(b)] $q^*$ is the unique Nash equilibrium of the game in which node $i$ chooses $q_i \in [0, C_i]$ to maximize $\ex[\pi_i]$ given $q_{-i}$.
\item[(c)] $q^*$ is the unique optimal solution of the linear program
\begin{align}\label{eq:lp}
  &\max \quad \ind^\top q \nonumber\\
  &\text{s.t.} \quad (I-P)q \le \alpha, \nonumber\\
  &\quad\quad\quad \bm 0\le q\le C.
\end{align}
\item[(d)] Let $U := \{i : \alpha_i + (Pq^*)_i > C_i\}$ and $\calh := \{1,\dots,N\} \setminus U$. Then $q^*_U = C_U$ and
\begin{equation}\label{eq:partition}
  q^*_\calh = (I_\calh - P_\calh)^{-1} (\alpha_\calh + P_{\calh,U} C_U).
\end{equation}
\end{enumerate}
The dual of \eqref{eq:lp} has optimal solution $y_\calh = \ind^\top (I_\calh - P_\calh)^{-1} \ge \ind$ on the demand constraints of $\calh$ and $w_U = \ind + y_\calh P_{\calh,U}$ on the capacity constraints of $U$; $y_i$ is the marginal effect on total network production of a unit increase in the order target of hedged node $i$.
\end{thm}

\begin{pro}[Computation and comparative statics]\label{prop:computation}
\begin{enumerate}
\item[(a)] Starting from $q^{(0)} = (I - P)^{-1}\alpha \ge q^*$, the iterates $T^n(q^{(0)})$ are nonincreasing and converge geometrically to $q^*$ at rate $\gamma$.
\item[(b)] $q^*$ is nondecreasing in $\alpha$ and in $C$.
\item[(c)] If the network is nondegenerate at $(\alpha, C)$, i.e.\ $\alpha_i + (Pq^*)_i \neq C_i$ for all $i$, then the partition $(\calh, U)$ is locally constant and $q^*$ is locally affine in $(\alpha, C)$.
\end{enumerate}
\end{pro}

\subsection{Covariance propagation and sensitivities}
\label{sec:covariance}

The static model takes the demand standard deviations $\sigma_i$ as given. We now derive them from a stochastic ordering process consistent with \eqref{eq:mean_demand}. Suppose each node operates a periodic base-stock policy with zero replenishment lead time, so that in each ordering round the order placed by node $j$ equals the demand it realized. Then the demand vector satisfies
\begin{equation}\label{eq:passthrough}
  D(s) = P D(s-1) + D^{\mathrm{ext}}(s),
\end{equation}
where $D^{\mathrm{ext}}(s)$ are i.i.d.\ with mean $v$ and covariance $\Sigma_D$, independent of $D(s-1)$.

\begin{thm}[Covariance propagation]\label{thm:variance}
Under Assumption \ref{defn:open} the process \eqref{eq:passthrough} has a unique stationary covariance $\Sigma$, which is the unique solution of the discrete Lyapunov equation
\begin{equation}\label{eq:lyapunov}
  \Sigma = P \Sigma P^\intercal + \Sigma_D, \qquad \text{equivalently} \qquad \Sigma = \sum_{k\ge0} P^k \Sigma_D (P^\intercal)^k.
\end{equation}
If $\Sigma_D = \operatorname{diag}(w_1^2, \dots, w_N^2)$, then $\sigma_i^2 := \Sigma_{ii}$ satisfies, exactly and for every topology,
\begin{equation}\label{eq:covariance_closed}
  \sigma^2 = N_\sigma w^2, \qquad N_\sigma := \sum_{k\ge0} P^k \odot P^k,
\end{equation}
where $\odot$ is the Hadamard product. $N_\sigma \ge M_\sigma := (I - P \odot P)^{-1}$ entrywise, with equality if and only if, for every pair of nodes and every $k$, at most one path of length $k$ between them has positive weight (in particular when inflows are arborescent); in that case \eqref{eq:covariance_closed} reduces to $\sigma^2 = M_\sigma w^2$.
\end{thm}

Two properties of \eqref{eq:covariance_closed} follow. First, $[N_\sigma]_{ii} \ge 1$ with equality iff $i$ lies on no directed cycle, so $\sigma_i^2 \ge w_i^2$. Second, since $p_{ij} \le 1$ implies $p_{ij}^2 \le p_{ij}$, an upstream node facing several customers sees pooled, not amplified, variability per unit of order: variability pools along tree paths and compounds along directed cycles.

\begin{pro}[Sensitivities]\label{prop:sens}
Let the network be nondegenerate at $(\alpha, C)$ with $\Sigma_D$ diagonal, and let $i \in \calh$, $j \in U$. Then, with $e_i$ the $i$-th unit vector,
\begin{align}
  \frac{\partial q^*_\calh}{\partial v_i} &= (I_\calh - P_\calh)^{-1} e_i, \label{eq:sens_v}\\
  \frac{\partial q^*_\calh}{\partial C_j} &= (I_\calh - P_\calh)^{-1} P_{\calh,U} e_j, \label{eq:sens_C}\\
  \frac{\partial \sigma}{\partial w_i} &= w_i \operatorname{diag}(1/\sigma)\, N_\sigma e_i, \label{eq:sens_sigma}\\
  \frac{\partial q^*_\calh}{\partial w_i} &= (I_\calh - P_\calh)^{-1} \operatorname{diag}(z_\calh)\, [\partial\sigma/\partial w_i]_\calh
  = w_i (I_\calh - P_\calh)^{-1} \operatorname{diag}(z_\calh/\sigma_\calh)\, [N_\sigma e_i]_\calh. \label{eq:sens_w}
\end{align}
All entries of \eqref{eq:sens_v}--\eqref{eq:sens_w} are nonnegative under Assumption \ref{defn:hedging}; \eqref{eq:sens_sigma}--\eqref{eq:sens_w} are exact and contain no omitted terms.
\end{pro}

The diagonal of the network multiplier is discussed in Remark~\ref{rem:self_multiplier} in Appendix~\ref{app:self_multiplier}.

We illustrate these concepts with an example in Appendix~\ref{app:example}.

We now embed the static clearing vector in a dynamic model of inventory drawdown after a capacity shock.

\section{Dynamic Model}
\label{sec:dynamic}
In practice, knowing the new equilibrium quantities is important, but so is the timing of how the network adapts to a shock. In this section we extend the model to a dynamic setting that identifies which nodes run short after a shock, when, and by how much. The standard way to incorporate a dynamic version is to solve a so-called \emph{Skorokhod Problem} (see \cite{skorokhod1961stochastic}). Our formulation follows the fluid networks of \textcite{chen2001fundamentals} and the dynamic clearing model of \textcite{chen2021financial}, with inventories in place of queues and unmet requirements in place of idleness.

A remark on orientation is in order. Orders propagate \emph{upstream}: the requirement of node $i$ aggregates the orders of its customers through the rows of $P$, exactly as in the static model. Shortfalls propagate \emph{downstream}: when a supplier cannot deliver, its customers bear the loss. Orders are therefore propagated by $P$, and shortfalls by the transpose of a matrix $\tilde P$ built from $P$ and the clearing orders.

\paragraph{Requirements.} Let $q$ be the clearing vector of Theorem \ref{thm:static_solution} at the current capacities $C$ (after a shock, the reduced capacities), and let
\begin{equation}\label{eq:flow}
  \lambda := \alpha + Pq, \qquad \text{so that} \qquad q = \lambda \wedge C \quad \text{and} \quad \lambda = \alpha + P(\lambda \wedge C).
\end{equation}
$\lambda_i$ is the rate at which node $i$ is required to deliver: $\alpha_i$ to its own market and $p_{ij} q_j$ to each customer $j$. Equation \eqref{eq:flow} is the \emph{flow equation}. Its solution is the static one, with $U = \{i : \lambda_i > C_i\}$ and $\calh$ its complement. The requirement includes the safety term $z_i\sigma_i$ (Remark~\ref{rem:stock_flow}): a node draws on its inventory whenever its throughput falls short of its requirement. Throughout this section and Section \ref{sec:shortage} we use only Assumption \ref{defn:open}, $\alpha \ge 0$, $C \ge 0$ and $\lambda > 0$; all of these hold for every supply chain network in the sense of Definition \ref{defn:scn}, where $\alpha > 0$.

The matrix $P$ records input shares. The dynamics need output shares:
\begin{equation}\label{eq:ptilde}
  \tilde P_{ij} := \frac{p_{ij}\, q_j}{\lambda_i},
\end{equation}
the fraction of node $i$'s requirement that is owed to customer $j$. The $i$-th row of $\tilde P$ sums to $1 - \alpha_i/\lambda_i$, the remainder being the share of node $i$'s own market.

Consider the evolution of the inventory if every supplier delivered in full, a so-called \emph{free inventory process}. Node $i$ would receive its throughput $q_i$ and release its requirement $\lambda_i$:
\begin{equation}\label{eq:free}
  X(t) = Z(0) + \theta^0 t, \qquad \theta^0 := (\lambda \wedge C) - \lambda,
\end{equation}
where $Z_i(0)$ is an initial level of inventory (not to be confused with the safety margin $z_i$). Thus $\theta^0_i = 0$ for a hedged node and $\theta^0_i = C_i - \lambda_i < 0$ for an unhedged node: no node accumulates stock, and unhedged nodes run a structural deficit. The requirement includes the safety term $z_i\sigma_i$; Remark~\ref{rem:stock_flow} in Appendix~\ref{app:stockflow} records the two readings of this term (a flow and a level), which coincide at $z = 0$.

Let $Z(t)$ be the actual inventory process. As the inventory cannot be negative, that is, $Z(t) \ge 0$, the Skorokhod problem trick is to create a process $Y(t)$, where $Y_i(t)$ is the cumulative requirement that node $i$ has failed to deliver. A node that runs short rations its own market and its customers pro rata, so each unit of node $i$'s shortfall removes $\tilde P_{ij}$ units from the inflow of customer $j$. As this must happen for all the nodes we have the relation
\begin{equation}\label{eq:reflection}
  Z_i(t) = X_i(t) + Y_i(t) - \textstyle\sum_j \tilde P_{ji} Y_j(t), \qquad \text{that is,} \qquad Z(t) = X(t) + (I - \tilde P^\intercal)\, Y(t).
\end{equation}
In this case, $Y_i(t)$ appears first as it cancels the part of $X_i(t)$ that would take the inventory of node $i$ negative, but the equation must also account for the shortfalls of the suppliers of node $i$, therefore the need for the term $-\sum_j \tilde P_{ji} Y_j(t)$. This keeps track of the propagation when other nodes have shortfalls, which is a mathematical way of tracking the contagion between nodes. When $Z_i(t) > 0$ there is still enough inventory to satisfy incoming orders and $dY_i(t) = 0$. All the conditions can be summarized in the program:
\begin{equation}\label{eq:skorokhod}
\begin{aligned}
  & Z(t) = X(t) + (I - \tilde P^\intercal)\, Y(t), \\
  & \text{s.t.} \quad Y(0) = \bm 0, \quad Y \text{ nondecreasing}, \quad Z(t) \ge \bm 0, \quad \int_0^\infty Z_i(t)\, dY_i(t) = 0 \ \text{ for every } i.
\end{aligned}
\end{equation}

In the notation of \textcite{chen2021financial}, \eqref{eq:skorokhod} is their dynamic clearing model with liability rates $\lambda$, relative liability matrix $\tilde P$, external inflows $(\ind-\varsigma)\odot q$ and initial liquid assets $Z(0)$; Theorem~\ref{thm:drawdown} is the counterpart of their Propositions 1--3 and also covers empty nodes and simultaneous exhaustions.

\subsection{Sequential dynamics}
\label{sec:sequential}
The solution is computed epoch by epoch. Let $E = \{i : Z_i = 0\}$ be the set of empty nodes and $\cala$ its complement, the stocked nodes, and write $L := \dot Y$ for the shortfall rates. While $E$ is fixed, empty nodes remain empty, so $\bm 0 = \theta^0_E + (I_E - \tilde P_{EE}^\intercal) L_E$, and stocked nodes have no shortfall:
\begin{equation}\label{eq:shortfall}
  L_E = (I_E - \tilde P_{EE}^\intercal)^{-1} (\lambda - q)_E, \qquad L_{\cala} = \bm 0.
\end{equation}
The stocked nodes follow $Z_j(t) = Z_j(0) + \theta_j t$ with the drawdown rates
\begin{equation}\label{eq:drawdown}
  \theta_j := \theta^0_j - \textstyle\sum_{i \in E} \tilde P_{ij} L_i \quad \text{for } j \in \cala, \qquad \theta_j := 0 \quad \text{for } j \in E;
\end{equation}
the term $-\sum_{i\in E} \tilde P_{ij} L_i$ charges the unmet requirements of empty suppliers to the inventories of their stocked customers. The system keeps evolving linearly according to these dynamics until the time $\tau$ where a stocked node runs out. This time is given by
\begin{equation}\label{eq:tau}
  \tau = \min_{j \in \cala,\ \theta_j < 0} \left\{ \frac{Z_j(0)}{-\theta_j} \right\},
\end{equation}
at which point the clock is restarted with the enlarged set $E$.

\begin{thm}[Inventory drawdown]\label{thm:drawdown}
Under Assumption \ref{defn:open} with $\alpha \ge 0$, $C \ge 0$ and $\lambda > 0$:
\begin{enumerate}
\item[(a)] \emph{Well-posedness.} $\tilde P \ge 0$ and $\rho(\tilde P) \le \rho(P) < 1$. Hence \eqref{eq:skorokhod} has a unique solution $(Z, Y)$ for every continuous $X$ with $X(0) \ge 0$, and the maps $X \mapsto Z$ and $X \mapsto Y$ are Lipschitz continuous in the uniform norm on every $[0, T]$.
\item[(b)] \emph{Fluid solution.} For $X(t) = Z(0) + \theta^0 t$ the solution is piecewise linear: between epochs $\dot Y = L$ and $\dot Z = \theta$, with $L$ and $\theta$ given by \eqref{eq:shortfall}--\eqref{eq:drawdown}. Moreover $0 \le L \le \lambda$.
\item[(c)] \emph{Monotone cascade.} $Z$ is nonincreasing, the empty set $E(t)$ is nondecreasing, and the shortfall rates $L(t)$ are nondecreasing. There are at most $N$ epochs, given by \eqref{eq:tau}.
\item[(d)] \emph{Finite-time limit.} After the last epoch,
\begin{equation}\label{eq:Linf}
  L = L^\infty := (I - \tilde P^\intercal)^{-1} (\lambda - q).
\end{equation}
\item[(e)] \emph{Nesting.} If no node is unhedged ($U = \emptyset$), then $\theta^0 = \bm 0$, $Y \equiv \bm 0$ and $Z(t) = Z(0)$: physical flows equal the static clearing vector at every $t$.
\end{enumerate}
\end{thm}

\begin{pro}[Exposure, long-run deliveries, and reserves]\label{prop:exposure}
Under the conditions of Theorem \ref{thm:drawdown}:
\begin{enumerate}
\item[(i)] \emph{Exposure.} $L^\infty_i > 0$ if and only if $i \in U$ or there is a chain of positive flows from a node of $U$ to $i$, that is, nodes $i_0 \in U, i_1, \dots, i_m = i$ with $p_{i_l i_{l+1}} q_{i_{l+1}} > 0$ for every $l$. A node with $L^\infty_i = 0$ never draws on its inventory. A node with $L^\infty_i > 0$ exhausts any finite inventory in finite time.
\item[(ii)] \emph{Long-run deliveries.} The long-run delivery rates $s^\infty := \lambda - L^\infty$ are nonnegative and solve
\begin{equation}\label{eq:longrun}
  s^\infty = (\ind - \varsigma) \odot q + \tilde P^\intercal s^\infty:
\end{equation}
in the long run every node delivers exactly what it produces and receives. The long-run fill rate of node $i$ is $s^\infty_i/\lambda_i$. Neither depends on $Z(0)$.
\item[(iii)] \emph{Reserves.} Let $\tau^{\mathrm{ex}}_i := \inf\{t : Z_i(t) = 0\}$. Increasing any coordinate of $Z(0)$ weakly decreases $Y(t)$ for every $t$ and weakly increases $\tau^{\mathrm{ex}}_i$ for every $i$.
\end{enumerate}
\end{pro}

Part (i) says that who is rationed is decided by the topology and the unhedged set; parts (ii)--(iii) say that initial inventories change the exhaustion times and not the long-run fill rates, and that a reserve held anywhere in the network never brings forward any node's exhaustion. Equation \eqref{eq:longrun} is the delivery-side counterpart of the flow equation \eqref{eq:flow}.

We summarize the exact simulation in Algorithm \ref{alg:sequential_dynamics}. Its line \ref{alg1:rates} evaluates \eqref{eq:ptilde}, \eqref{eq:shortfall} and \eqref{eq:drawdown}.

\begin{algorithm}[!htpb]
\caption{Sequential Dynamics Simulation for Supply Chain Networks}
\label{alg:sequential_dynamics}
\begin{algorithmic}[1]
\Require{Initial inventory vector $Z(0) \ge \bm 0$, routing matrix $P$, capacity vector $C$, parameters $\alpha = v + z \odot \sigma$, total simulation time $T$.}
\Ensure{Inventory trajectory $Z(t)$ and effective production flows $\lambda(t)$ for $t \in [0, T]$.}
\State $t \gets 0$ \Comment{Initialize simulation clock}
\State $Z \gets Z(0)$ \Comment{Set initial inventory levels}
\While{$t < T$}
    \State Identify sets $\calh$ (hedged) and $U$ (unhedged) from the clearing vector $q$, with $\lambda=\alpha+Pq$ (Theorem \ref{thm:static_solution}): \label{alg1:partition}
    \State \quad $\calh \gets \{ i \in \{1,\dots,N\} \mid \lambda_i \le C_i \}$
    \State \quad $U \gets \{1,\dots,N\} \setminus \calh$
    \State Compute flows for hedged nodes:
    \State \quad $\lambda_\calh \gets (I_\calh - P_\calh)^{-1} (\alpha_\calh + P_{\calh,U} C_U)$ \label{alg1:hedged}
    \State Compute flows for unhedged nodes:
    \State \quad $\lambda_U \gets \alpha_U + P_{U,\calh} \lambda_\calh + P_{U,U} C_U$ \label{alg1:unhedged}
    \State Calculate drawdown rates: $E \gets \{i : Z_i = 0\}$; $\tilde P_{ij} \gets p_{ij} (\lambda_j \wedge C_j)/\lambda_i$; $L_E \gets (I_E - \tilde P_{EE}^\intercal)^{-1} (\lambda - \lambda \wedge C)_E$; $\theta_j \gets \big( (\lambda_j \wedge C_j) - \lambda_j - \sum_{i \in E} \tilde P_{ij} L_i \big)\, \ind\{Z_j > 0\}$ \label{alg1:rates}
    \State $\tau \gets \infty$
    \For{each $j$ with $Z_j > 0$}
        \If{$\theta_j < 0$}
            \State $\tau \gets \min\{\tau,\ Z_j/(-\theta_j)\}$ \Comment{Time until node $j$ hits zero inventory} \label{alg1:tau}
        \EndIf
    \EndFor
    \If{$t + \tau > T$}
        \State $\tau \gets T - t$ \Comment{Cap the final step at total duration $T$}
    \EndIf
    \State Update inventory for $t \in [t, t+\tau]$: $Z \gets Z + \theta\,\tau$ \Comment{only stocked deficit nodes draw down}
    \State $t \gets t + \tau$ \Comment{Advance simulation clock}
    \If{$\tau = \infty$ or $t \ge T$}
        \State \textbf{break} \Comment{System has stabilized or time limit reached} \label{alg1:break}
    \EndIf
\EndWhile
\end{algorithmic}
\end{algorithm}

The computational cost of the static equilibrium, the drawdown path, the long-run exposure, and the priced variant is summarized in Proposition~\ref{prop:tractable} in Appendix~\ref{app:computation_networks}.

\section{Order Inflation, Rebalancing, and Price Feedback}
\label{sec:pricing}
When a node sees its upstream supply decrease, it will try to increase its orders from its existing suppliers or to look for new ones, rather than wait out the shock. In this section we model both responses. Inflating orders redistributes a shortage and can spread it to suppliers with no physical shortage (Section \ref{sec:shortage}); rebalancing and price feedback are treated in Section \ref{sec:rebalancing}. Mathematically, we make the routing matrix dynamic, that is, $P(t)$ evolves over time, and we show the relation to pricing inside the network via the safety stock multipliers $z_i$.

\subsection{Order inflation under rationing: the shortage game}
\label{sec:shortage}
When a supplier rations pro rata, a customer that orders more receives more. We ask what orders result when every node anticipates this. The game is posed for the long-run regime of Proposition \ref{prop:exposure}, in which exposed nodes have exhausted their inventories and deliveries are stationary.

Each node $j$ chooses an order level $x_j \in [0, C_j]$: it requests $p_{ij} x_j$ from each supplier $i$ and plans own production $(1 - \varsigma_j) x_j$. Orders are free; a node pays only for what it receives. Given a profile $x$, the requirement of node $i$ is $\lambda_i(x)$, with $\lambda(x) := \alpha + Px$, and node $i$ serves its own market and its customers at a common fill rate $\psi_i \in [0, 1]$. Node $j$ therefore receives $x_j \Psi_j(\psi)$, where
\begin{equation}\label{eq:Psi}
  \Psi_j(\psi) := 1 - \varsigma_j + \textstyle\sum_i p_{ij} \psi_i, \qquad \text{that is,} \qquad \Psi(\psi) = \ind - \varsigma + P^\intercal \psi,
\end{equation}
is the share-weighted fill rate of its sources. Node $j$'s expected profit is the Newsvendor objective of Section \ref{sec:model} evaluated at its receipts; it is single-peaked with peak at the requirement $\lambda_j(x)$. Taking as given the orders it receives and the fill rates of its suppliers, node $j$'s best order is its requirement divided by its fill rate, capped at capacity:
\begin{equation}\label{eq:best_order}
  x_j = \min\{ C_j,\ \lambda_j(x)/\Psi_j(\psi) \}.
\end{equation}

\begin{defn}[Shortage equilibrium]\label{defn:shortage}
A \emph{shortage equilibrium} is a pair $(x, \psi) \in [0, C] \times [0, 1]^N$ such that, for every node $j$,
\begin{itemize}
\item[(SE1)] consistency: $\psi_j \lambda_j(x) = \min\{ \lambda_j(x),\ x_j \Psi_j(\psi) \}$, that is, node $j$ delivers what it receives, up to its requirement;
\item[(SE2)] optimality: $x_j \Psi_j(\psi) \le \lambda_j(x)$, with equality whenever $x_j < C_j$, that is, no node receives more than its requirement, and a node that receives less is ordering at capacity.
\end{itemize}
\end{defn}

(SE2) is \eqref{eq:best_order} written without division. Fill-rate expectations are fulfilled: each node responds optimally to fill rates that its own and the other nodes' orders generate. \emph{Truthful ordering} is $x = q$, the clearing vector of Theorem \ref{thm:static_solution}, for which $\lambda(q)$ is the requirement vector $\lambda$ of Section \ref{sec:dynamic}; by Proposition \ref{prop:exposure}(ii) its fill rates are $\psi^T_i := s^\infty_i/\lambda_i$. Recall $U = \{i : \lambda_i > C_i\}$.
In a shortage equilibrium $(x,\psi)$ we also call node $i$ unhedged if $\lambda_i(x) > C_i$; by Theorem~\ref{thm:shortage}(a) every node of $U$ is unhedged in every equilibrium.

\begin{thm}[Shortage game]\label{thm:shortage}
Under Assumption \ref{defn:open} with $\alpha \ge 0$, $C \ge 0$ and $\lambda > 0$:
\begin{enumerate}
\item[(a)] \emph{Order inflation.} Every shortage equilibrium has $x \ge q$ and hence $\lambda(x) \ge \lambda$. Every node that orders at capacity under truthful ordering orders at capacity in equilibrium.
\item[(b)] \emph{Existence and lattice structure.} Shortage equilibria exist. Ordered by ``higher orders and lower fill rates'', they have a least element and a greatest element. The least equilibrium is the limit of the iteration $x_j \leftarrow \min\{C_j, \lambda_j(x)/\Psi_j(\psi)\}$, $\psi_j \leftarrow \min\{1, C_j \Psi_j(\psi)/\lambda_j(x)\}$ (simultaneous updates) started from truthful orders and full fill rates $(x, \psi) = (q, \ind)$; along the iteration orders rise and fill rates fall.
\item[(c)] \emph{Physical bottlenecks persist.} In every equilibrium each $j \in U$ rations, with $\psi_j \le C_j/\lambda_j(x) \le C_j/\lambda_j < 1$. If $j \in U$ has no suppliers inside the network ($\varsigma_j = 0$), then $C_j/\lambda_j = \psi^T_j$: its fill rate is no higher than under truthful ordering.
\item[(d)] \emph{Truthful ordering.} There is an equilibrium with $x = q$ if and only if $\Psi_j(\psi^T) = 1$ for every node $j$ with $q_j < C_j$, that is, if and only if no node with spare capacity has a rationing supplier. In particular truthful ordering is an equilibrium when $U = \emptyset$.
\item[(e)] \emph{Strategic bottlenecks.} A node outside $U$ can order at capacity and ration in equilibrium, and a node that is not downstream of $U$, hence not exposed in the sense of Proposition \ref{prop:exposure}(i), can be rationed (Example \ref{exm:bottleneck}). Theorem \ref{thm:network}(e) gives a test.
\end{enumerate}
\end{thm}

\begin{pro}[Single bottleneck]\label{prop:bottleneck}
Consider a star network: a supplier $i$ and customers $j \in J$ with $p_{ij} \in (0, 1]$ and no other links, with $\alpha > 0$, $C > 0$, and $\lambda_i > C_i$. Write $\Psi_j(\psi) = 1 - p_{ij} + p_{ij}\psi$ for a scalar $\psi$.
\begin{enumerate}
\item[(1)] The shortage equilibrium is unique. The supplier orders $x_i = C_i$ and its fill rate $\psi^* \in (0, 1)$ is the unique root of
\[
  \psi \Big[ \alpha_i + \sum_{j \in J} p_{ij} \min\big\{ C_j,\ \alpha_j/\Psi_j(\psi) \big\} \Big] = C_i,
\]
and each customer orders $x_j = \min\{C_j, \alpha_j/\Psi_j(\psi^*)\}$.
\item[(2)] Customer $j$ receives $\min\{\alpha_j, C_j \Psi_j(\psi^*)\}$: it is made whole if and only if $C_j/\alpha_j \ge 1/\Psi_j(\psi^*)$. The shortage is allocated by capacity, capped at need.
\item[(3)] $\psi^* \le \psi^T_i = C_i/\lambda_i$, with strict inequality if and only if some customer has $C_j > \alpha_j$. The supplier's own market and every customer with $C_j \le \alpha_j$ receive no more than under truthful ordering.
\item[(4)] The shortage equilibrium is the unique Nash equilibrium of the game in which every node chooses $x_j \in [0, C_j]$, fill rates are determined by (SE1), and each node's payoff is its Newsvendor expected profit at its receipts, each node accounting for the effect of its own order on the fill rate.
\end{enumerate}
\end{pro}

\begin{exm}[{\bf A strategic bottleneck}]\label{exm:bottleneck}
Suppliers 1 and 2, customers 3 and 4. Customer 3 sources half from each supplier and customer 4 sources only from supplier 2: $p_{13} = p_{23} = 1/2$, $p_{24} = 1$, all other entries zero. Let $\alpha = (0, 0, 20, 10)$ and $C = (5, 22.5, 50, 10)$; supplier 1 is the shocked node.

\noindent\emph{Truthful ordering:} $q = (5, 20, 20, 10)$ and $\lambda = (10, 20, 20, 10)$, so $U = \{1\}$. Supplier 1 fills half of its requirement and supplier 2, with 2.5 units of spare capacity, fills all of its own. Customer 3 receives 15 of its 20; customer 4 receives its full 10.

\noindent\emph{Shortage equilibrium (unique):} $x = (5, 22.5, 40, 10)$ with $\psi_1 = 1/4$ and $\psi_2 = 3/4$. Customer 3 doubles its order and is made whole ($40 \times \tfrac12(\tfrac14 + \tfrac34) = 20$). Its inflated order raises the requirement of supplier 2 from 20 to 30, above its capacity of 22.5: supplier 2 becomes unhedged although it suffered no shock. Customer 4, which has no link to supplier 1, receives 7.5 instead of 10. Total deliveries rise from 25 to 27.5 because the spare capacity of supplier 2 is drawn into use, but a quarter of customer 4's supply is transferred to customer 3.

\noindent The bottleneck at supplier 2 is predicted by the test of Theorem \ref{thm:network}(e): customer 3 has slack and buys from the unhedged supplier 1, so $\Psi_3(\chi) = 3/4$ and $\hat\lambda_2 = \tfrac12 \min\{50, 20/(3/4)\} + 10 = 23.3 > 22.5 = C_2$. Total served demand, $20 \cdot 1 + 10 \cdot \tfrac34 = 27.5$, equals total own production $x_1 + x_2$, as in Theorem \ref{thm:network}(a).
\end{exm}

\begin{thm}[Shortage game on a network]\label{thm:network}
Under Assumption \ref{defn:open} with $\alpha \ge 0$, $C \ge 0$ and $\lambda > 0$:
\begin{enumerate}
\item[(a)] \emph{Conservation.} In every shortage equilibrium $(x, \psi)$,
\begin{equation}\label{eq:conservation}
  \sum_j \alpha_j \psi_j = \sum_j (1 - \varsigma_j)\, x_j:
\end{equation}
the total own-market requirement that is served equals total own production.
\item[(b)] \emph{Aggregate deliveries.} In every shortage equilibrium $\sum_j \alpha_j \psi_j \ge \sum_j \alpha_j \psi^T_j$, with strict inequality if and only if $x_j > q_j$ for some node $j$ with own production ($\varsigma_j < 1$). Order inflation never lowers the total that is delivered; it changes who receives it.
\item[(c)] \emph{Invariance.} All shortage equilibria have the same fill rate $\psi_j$ at every node with $\alpha_j > 0$, the same order $x_j$ at every node with $\varsigma_j < 1$, and the same aggregate \eqref{eq:conservation}.
\item[(d)] \emph{Uniqueness.} If $\alpha > 0$, in particular for every supply chain network in the sense of Definition \ref{defn:scn}, the shortage equilibrium is unique.
\item[(e)] \emph{Strategic bottlenecks.} Let $\chi_k := \min\{1, C_k/\lambda_k\}$, so that $\chi_k < 1$ exactly for $k \in U$, and let
\[
  \hat\lambda_i := \alpha_i + \sum_j p_{ij} \min\big\{ C_j,\ \lambda_j/\Psi_j(\chi) \big\}, \qquad \Psi_j(\chi) = 1 - \sum_{k \in U} p_{kj} \big( 1 - C_k/\lambda_k \big),
\]
with $\lambda_j/0 := \infty$. If $\hat\lambda_i > C_i$, then node $i$ rations ($\psi_i < 1$) and orders at capacity ($x_i = C_i$) in every shortage equilibrium. For $i \notin U$ this happens only if $i$ has a customer $j$ with spare capacity ($C_j > \lambda_j$) that also buys from a node of $U$. Conversely, if $\alpha_i + \sum_j p_{ij} C_j \le (1 - \varsigma_i) C_i$, then $\psi_i = 1$ in every shortage equilibrium.
\item[(f)] \emph{Computation.} Let $\Theta$ be the simultaneous update of Theorem \ref{thm:shortage}(b), $\zeta^n := \Theta^n(q, \ind)$ and $\zeta_n := \Theta^n(C, \bm 0)$. Along $\zeta^n$ orders rise and fill rates fall, and $\zeta^n$ converges to the least equilibrium; along $\zeta_n$ orders fall and fill rates rise, and $\zeta_n$ converges to the greatest equilibrium. Every equilibrium lies between $\zeta^n$ and $\zeta_n$ for every $n$. One sweep costs two products of a vector with $P$, that is, $O(N + m)$ operations. If $\alpha > 0$ both sequences converge to the unique equilibrium, and their gap bounds the error.
\end{enumerate}
\end{thm}

Parts (a)--(b) qualify the usual view of shortage gaming: with fixed sourcing shares, inflated orders draw idle capacity into use, so the aggregate cannot fall, while part (e) shows that the cost is borne by the customers of suppliers that the inflated orders push to capacity. In contrast with Proposition \ref{prop:exposure}(i), where exposure runs only downstream of $U$, strategic exposure also runs sideways: from a rationed customer to its other suppliers, and from them to their other customers.

\begin{pro}[Nash equilibrium on a network]\label{prop:nash}
Let Assumption~\ref{defn:open} hold with $\alpha>0$, $C>0$ and $p_{jj}=0$ for every $j$. Consider the game in which every node $j$ chooses $x_j\in[0,C_j]$, fill rates are determined by (SE1), and the payoff of node $j$ is its Newsvendor expected profit at its receipts, each node accounting for the effect of its own order on all fill rates.
\begin{enumerate}
\item[(a)] \emph{Fill rates.} For every order profile $x\in[0,C]$, (SE1) has a unique solution $\psi(x)$, and $\psi(x)$ is continuous in $x$.
\item[(b)] \emph{Monotone receipts.} The receipts of node $j$, $R_j(x):=x_j\Psi_j(\psi(x))$, are continuous and nondecreasing in its own order $x_j$. For fixed orders of the other nodes, they are either zero for every $x_j$ or strictly increasing in $x_j$.
\item[(c)] \emph{Nash equilibrium.} A profile $x$ is a Nash equilibrium if and only if $(x,\psi(x))$ is a shortage equilibrium. By Theorem~\ref{thm:network}(d), the Nash equilibrium exists and is unique.
\end{enumerate}
\end{pro}

Part (b) is the reason for part (c). A node cannot lower its receipts by ordering more, so the rule ``order until receipts meet the requirement or the order cap binds'' is optimal whether or not the node takes fill rates as given. The condition $p_{jj}=0$ ensures that the requirement of a node does not depend on its own order. The scope of the equilibrium concept and the open questions it leaves are recorded in Remark~\ref{rem:shortage_scope} in Appendix~\ref{app:scope}.

\subsection{Rebalancing and price feedback}
\label{sec:rebalancing}

In practice, when there is a shock two natural behaviors occur at the same time: on one hand, every node that sees their supply reduced will try to rebalance their orders so that they can mitigate the effect of the initial reduction. On the other hand, a decrease of the overall supply will increase the prices and affect the demand.

\textbf{Routing Matrix Rebalancing:} To handle the first effect, the dynamic routing matrix $P(t)$ changes at the epochs $\tau_1,\tau_2,\dots$, when the inventory of a set of nodes becomes exhausted. At this time, the nodes restructure their trade proportions as:
\begin{eqnarray}
  \label{eq:rebalancing}
  p_{ij}(\tau_k)=\ind\{i\in\cala_k\}\;\varsigma_j\,
  \frac{p_{ij}(\tau_{k-1})}{\sum_{m\in\cala_k} p_{mj}(\tau_{k-1})},
  \qquad \varsigma_j:=\sum_{m=1}^N p_{mj}(0),
\end{eqnarray}
where $\mathcal{A}_k := \{j : L_j(\tau_k) = 0\}$ is the set of nodes that deliver in full at $\tau_k$, with $L$ the shortfall rates of Section~\ref{sec:sequential} evaluated at the routing matrix $P(\tau_{k-1})$ and the set of empty nodes at $\tau_k$. This rule redistributes the sourcing shares of rationing suppliers proportionally among those that deliver in full while holding each column's total sourcing mass fixed at its initial level $\varsigma_j$. A node without stock that receives and delivers its full requirement, such as a transit node, remains active.
\begin{lem}\label{lem:rebalancing}
Suppose that at every epoch each column $j$ with $\varsigma_j > 0$ retains at least one active supplier, so that \eqref{eq:rebalancing} is well defined (columns with $\varsigma_j = 0$ stay zero). Then $\sum_i p_{ij}(\tau_k) = \varsigma_j$ for every $j$ and $k$. If, in addition, every node of $\cala_k$ reaches, through its chain of active suppliers, a node $j \in \cala_k$ with $\varsigma_j < 1$ (which holds automatically when $\varsigma_j < 1$ for every $j$), then $\rho(P(\tau_k)) < 1$, $I - P(\tau_k)$ is a nonsingular $M$-matrix, and Theorem \ref{thm:static_solution} applies with $P(\tau_k)$ in place of $P$.
\end{lem}
We assume the conditions of Lemma \ref{lem:rebalancing} throughout. After each rebalancing, $\lambda$, $\tilde P$ and the drawdown rates are recomputed and the Skorokhod problem \eqref{eq:skorokhod} is restarted from the current inventories; Theorem \ref{thm:drawdown}(a)--(b) apply on each inter-epoch interval, while the monotonicity statements of Theorem \ref{thm:drawdown}(c) and Proposition \ref{prop:exposure} are for fixed data.

\textbf{Price Elasticity Effects:} Likewise, with every failure epoch $\tau_k$ the delivered supply of the network is reduced. The market-relevant throughput is the \emph{delivered supply}: the part of each sovereign node's own requirement that is actually served, whether from flow or from inventory,
\begin{equation}\label{eq:served}
  Q(t) = \sum_{i \in \cals} \alpha_i(t) \left( 1 - \frac{L_i(t)}{\lambda_i(t)} \right),
\end{equation}
where $\cals$ is the set of sovereign (non-transit) nodes and $L$ are the shortfall rates of \eqref{eq:shortfall}. A stocked node has $L_j = 0$ and serves its requirement in full from its buffer; its shortfall enters $Q(t)$ only from its exhaustion epoch onward. Consequently $p = p_0$ remains a fixed point of the relation below at the shock onset whenever every sovereign node is stocked, and $Q(t)$ changes only at the epochs $\tau_k$. Let $\varepsilon>0$ and $\eta\ge0$ denote the short-run price elasticities of demand and supply. Clearing the demand curve $Q_d(p)=Q(0)(p/p_0)^{-\varepsilon}$ against the supply response $Q_s(p)\propto p^{\eta}$ yields the incidence relation
\begin{eqnarray}
  \label{eq:price}
  p(t)=p_0\left(\frac{Q(0)}{Q(t)}\right)^{1/(\varepsilon+\eta)},
\end{eqnarray}
where $Q(0)$ is the pre-shock equilibrium delivered supply (so that $p\equiv p_0$ absent a shock) and $\eta=0$ recovers the perfectly-inelastic supply stress case. Consistency with the same demand curve gives the external demand response
\begin{eqnarray}
  \label{eq:demand_change}
  v_i(t)=v_i(0)\left(\frac{p_0}{p(t)}\right)^{\!\varepsilon}.
\end{eqnarray}
The critical ratios respond through the node prices, which move with the index, $p_i(t)=p_i(0)\,p(t)/p_0$. We anchor the unit costs to the baseline safety level $\bar z$ of the static model, $c_i=p_i(0)(1-\Phi(\bar z))$, so that $z_i(p_0)=\bar z$ exactly and the model with prices nests Algorithm \ref{alg:sequential_dynamics} at baseline prices:
\begin{eqnarray}
  \label{eq:safety}
  z_i(t)=\Phi^{-1}\!\big(\rho_i(t)\big)=\Phi^{-1}\!\Big(1-(1-\Phi(\bar z))\,\frac{p_0}{p(t)}\Big).
\end{eqnarray}
Finally, $\alpha_i(t)=\max\{0,\,v_i(t)+z_i(t)\sigma_i\}$, the floor mirroring the primal constraint $q\ge\bm0$. Nodes outside $\mathcal{S}$ have no market of their own, and we set $\alpha_i(t)=0$ for $i\notin\mathcal{S}$.

\begin{pro}[Price equilibrium]\label{prop:price_equilibrium}
Fix an epoch with empty set $E$. Let $Q(p)$ be \eqref{eq:served} evaluated at $\alpha(p) = \max\{0, v(p) + z(p) \odot \sigma\}$, and let $F(p) := p_0 (Q(0)/Q(p))^{1/(\varepsilon+\eta)}$, projected onto $[p_0/2, \kappa p_0]$, with $F(p) := \kappa p_0$ where $Q(p) = 0$. Suppose $\lambda(p) > 0$ for every $p \in [p_0/2, \kappa p_0]$.
\begin{enumerate}
\item[(a)] \emph{Existence.} $F$ is continuous, so a fixed point exists.
\item[(b)] \emph{Viability.} Where $\alpha_i(p) > 0$, the order target of node $i$ falls as the price rises if and only if the viability inequality $\varepsilon\, v_i(p)\, \phi(z_i(p))\, p_i(p) \ge c_i \sigma_i$ holds: demand destruction then dominates safety-stock accumulation at node $i$, and the reverse holds when the inequality fails.
\item[(c)] \emph{Uniqueness.} If $h(p) := p^{\varepsilon+\eta} Q(p)$ is strictly increasing on $[p_0/2, \kappa p_0]$, the fixed point is unique. This holds whenever no sovereign node is empty ($E \cap \cals = \emptyset$), $v_i(0) \ge 0$ for every $i \in \cals$, $\Phi(\bar z) > 3/4$, and either $\eta > 0$ with $v_i(0) > 0$ for some $i \in \cals$, or $\sigma_i > 0$ for some $i \in \cals$. For an arbitrary empty set $E$, the fixed point is unique if $Q(0)>0$, $v_i(0)\ge 0$ for every $i\in\mathcal{S}$, $\Phi(\bar z)>3/4$, $\alpha_i\equiv 0$ for every $i\notin\mathcal{S}$, and either $\eta>0$, or $\sigma_i>0$ for every $i\in\mathcal{S}$ with $v_i(0)>0$. Under these conditions $p^{\varepsilon}Q(p)$ is nondecreasing in $p$: a price increase lowers delivered supply by no more than it lowers demand.
\item[(d)] \emph{Damped iteration.} If $F$ is nonincreasing and Lipschitz with constant $L_F$ on the interval, the damped iteration $p \leftarrow (1 - \omega)p + \omega F(p)$ with $\omega \in (0, 2/(1+L_F))$ converges to the unique fixed point.
\item[(e)] \emph{Safety-stock channel.} Holding $v$ fixed, $z(p)$ and hence $\alpha(p)$ are nondecreasing in $p$, and the unhedged set $U(\alpha(p), C)$ is nondecreasing in $p$.
\end{enumerate}
\end{pro}

These changes are reflected in Algorithm \ref{alg:enhanced_strategic_dynamics}. The next section applies the drawdown model of Section~\ref{sec:dynamic} and the price feedback of this section to the global crude oil trade network. Section~\ref{sec:exp_game} applies the shortage game of Section~\ref{sec:shortage} to the same network. Rebalancing is not used in Section~\ref{sec:data_and_experiment}.

\begin{algorithm}[!htpb]
\caption{Sequential dynamics with price feedback and rebalancing}
\label{alg:enhanced_strategic_dynamics}
\begin{algorithmic}[1]
\Require{Initial inventory $Z(0)\ge\bm 0$; initial routing matrix $P(0)$ with column masses $\varsigma_j$; capacities $C$; baseline demand $v(0)$; volatilities $\sigma$; baseline price $p_0$; baseline safety level $\bar z$; elasticities $\varepsilon,\eta$; price cap $\kappa$; sovereign set $\cals$; horizon $T$.}
\Ensure{Trajectories $Z(t)$, $p(t)$, $P(t)$, $\lambda(t)$ over $t\in[0,T]$.}
\State $t\gets0$;\quad $Z\gets Z(0)$;\quad $P\gets P(0)$;\quad $p\gets p_0$
\State $Q(0)\gets$ delivered supply \eqref{eq:served} at $p_0$ and pre-shock capacities \Comment{pre-shock delivered supply at $p_0$}
\While{$t<T$}
  \State Solve the epoch fixed point by damped iteration on $p\in[p_0/2,\,\kappa p_0]$:
  \State \quad $z(p)\gets\Phi^{-1}\!\big(1-(1-\Phi(\bar z))\,p_0/p\big)$;\;\; $v_i(p)\gets v_i(0)(p_0/p)^{\varepsilon}$;\;\; $\alpha_i(p)\gets\max\{0,\,v_i(p)+z(p)\sigma_i\}$ \label{alg2:targets}
  \State \quad $\lambda(p)\gets$ requirements at $\alpha(p)$ via lines \ref{alg1:partition}--\ref{alg1:unhedged} of Algorithm \ref{alg:sequential_dynamics}
  \State \quad $Q(p)\gets$ delivered supply \eqref{eq:served} at $\alpha(p)$, with $L$ from line \ref{alg1:rates} of Algorithm \ref{alg:sequential_dynamics};\;\; $p\gets p_0(Q(0)/Q(p))^{1/(\varepsilon+\eta)}$ (damped, projected) \label{alg2:price}
  \State Fix $(\lambda,\alpha)$ at the converged $p$;\quad $\theta\gets$ drawdown rates of line \ref{alg1:rates} of Algorithm \ref{alg:sequential_dynamics}
  \State $\tau\gets\min_{j:\theta_j<0}\{Z_j/(-\theta_j)\}$ (set $\tau\gets\infty$ if none);\quad \textbf{if } $t+\tau>T$ \textbf{ then } $\tau\gets T-t$
  \State $Z\gets Z+\theta\,\tau$;\quad $t\gets t+\tau$
  \State Rebalance ($L$ recomputed at the updated $Z$): $\cala\gets\{j:L_j=0\}$;\quad $p_{ij}\gets\ind\{i\in\cala\}\,\varsigma_j\,p_{ij}\big/\!\sum_{m\in\cala}p_{mj}$ whenever $\sum_{m\in\cala}p_{mj}>0$
  \State \textbf{if } $\tau=\infty$ or $t\ge T$ \textbf{ then break}
\EndWhile
\end{algorithmic}
\end{algorithm}
\section{Numerical Experiment: A Closure of the Strait of Hormuz}\label{sec:data_and_experiment}

We apply the model to the world crude oil trade network and study a sustained loss of 90\% of the capacity of the Strait of Hormuz. The experiment has two purposes. The first is to show that the clearing vector and the exact stock path can be computed on a network of realistic size. The second is to measure which importers are exposed to the chokepoint, how much of their supply they lose, and how long their stocks last. A further exercise, in Section~\ref{sec:exp_game}, studies order inflation by the most exposed importers. Inside the model all quantities are in million barrels (MMbbl) per week; we report flows in million barrels per day (mb/d).

\subsection{Data}\label{sec:exp_data}

The trade network comes from monthly bilateral records for crude petroleum (HS 2709) in UN Comtrade, November 2021 to December 2025; Appendix~\ref{app:data} describes the construction of the flows, the transit nodes, the model inputs, and the initial stocks. The network has 163 active nodes and carries 46.6 mb/d of crude. Three transit nodes represent the Strait of Hormuz, the Strait of Malacca, and the Suez Canal with the SUMED pipeline, and all exports of Saudi Arabia, the United Arab Emirates, Iraq, Kuwait, Qatar, Bahrain and Iran are routed through the Hormuz node, which carries 15.8 mb/d, or 34\% of the crude in the network. Figure~\ref{fig:exp_network} shows the transit nodes and the eight importers discussed in Section~\ref{sec:exp_results}.

\begin{figure}[ht]
\centering
\includegraphics[width=\textwidth]{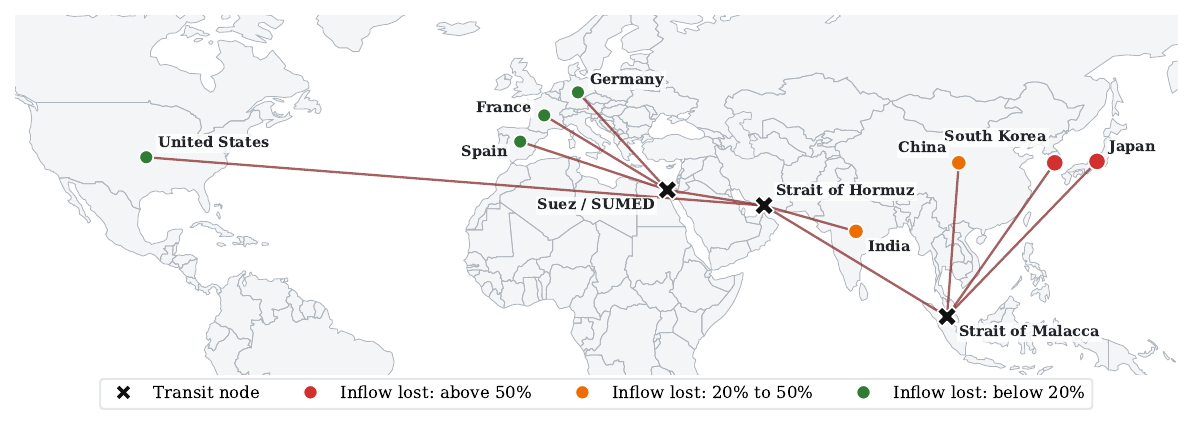}
\caption{Transit nodes and the eight importers of Table~\ref{tab:downstream_flow_comparison}. Lines are the links of the model network from the transit nodes to the displayed nodes; they are not shipping lanes. Colours give the share of required inflow lost under the closure: above 50\%, between 20\% and 50\%, and below 20\%.}\label{fig:exp_network}
\end{figure}

\subsection{Design}\label{sec:exp_design}

Each node orders its net demand plus the safety term $z\sigma_i$. We take $z=0$, at which nodes order exactly their net demand, as the physical baseline (Remark~\ref{rem:stock_flow}), and $z=1.645$ and $z=\bar z=2.576$ as precautionary scenarios. At $z=0$ Assumption~\ref{defn:hedging} does not hold. This does not affect the experiment: the proof of Theorem~\ref{thm:static_solution}(a) uses only $\alpha\ge 0$, and the results of Section~\ref{sec:dynamic} require only Assumption~\ref{defn:open} with $\alpha\ge 0$, $C\ge 0$ and $\lambda>0$, which hold on the active network. Before the shock no node is unhedged, there is no shortfall, and all stocks are constant. At time zero the capacity of the Hormuz node is cut to 10\% of its baseline requirement and stays there for the 52-week horizon. No other capacity and no routing proportion is changed. The losses reported below are therefore produced by the clearing mechanism and are not imposed node by node. The stock path is computed exactly by the sequential algorithm of Section~\ref{sec:dynamic}. In a variant at $\bar z$ the price responds to the shortage as in Section~\ref{sec:rebalancing}. Because $\eta>0$, the price equilibrium of every epoch is unique (Proposition~\ref{prop:price_equilibrium}(c)). Importers do not switch suppliers in any run: the rebalancing step of Section~\ref{sec:rebalancing} is switched off, so the routing matrix is fixed.

Every run verifies the implementation. The code reproduces the worked examples of the paper. Clearing at $z=0$ reproduces the observed flows. In the long run, served net demand equals production to numerical precision. Along the path, stocks are nonnegative, the shortfall is zero wherever stock is positive, shortfall rates are nondecreasing in time and never exceed their long-run levels, and a positive shortfall occurs only at nodes downstream of the unhedged node. Computational details and checks for this section and Section~\ref{sec:exp_game} are collected in Appendix~\ref{app:computation}.

\subsection{Results}\label{sec:exp_results}

\paragraph{Exposure.} Table~\ref{tab:downstream_flow_comparison} reports the share of its required inflow that each of eight importers loses while it still holds stock. The lower value applies from the moment of the closure, when every other node still ships in full from stock. The upper value applies once all other stocks in the network are exhausted. Japan loses 84\% of its crude supply, South Korea 63--65\%, India 43--44\% and China 33--37\%. Because the closure removes 90\% of the strait's throughput, these numbers mean that 93\% of Japan's crude, 70\% of South Korea's and 47\% of India's passes through the strait in the data. The European importers and the United States lose between 5\% and 12\%. Second-round losses through re-exporting countries are small: the two values differ by at most 3.7 percentage points. The loss shares are the same at all three safety levels, because the routing proportions are fixed and the closure is defined relative to the chokepoint's own baseline throughput. Across the network, 88 of the 163 nodes lose more than 1\% of their required inflow.

\begin{table}[ht]
\centering
\caption{Exposure to a 90\% closure of the Strait of Hormuz. Loss is the share of required inflow lost while the node holds stock: at the closure, and after all other stocks are exhausted. The last two columns give the number of weeks for which 30 days of net-demand cover last; the range corresponds to the two loss shares. For the United States the required inflow includes crude that it re-exports.}\label{tab:downstream_flow_comparison}
\begin{tabular}{lrrrrr}
\hline
 & Net demand & \multicolumn{2}{c}{Loss (\%)} & \multicolumn{2}{c}{Weeks per 30 days of cover} \\
Node & (mb/d) & at closure & others exhausted & $z=0$ & $z=\bar z$ \\
\hline
Japan         &  2.31 & 83.7 & 84.3 & 5.1        & 3.9--4.0 \\
South Korea   &  2.75 & 62.7 & 64.8 & 6.6--6.8   & 5.5--5.7 \\
India         &  5.28 & 42.7 & 43.8 & 9.8--10.0  & 7.3--7.5 \\
China         & 11.10 & 33.2 & 36.9 & 11.6--12.9 & 9.4--10.4 \\
Spain         &  1.22 &  9.5 & 11.8 & 36.3--45.2 & 30.4--37.8 \\
France        &  1.00 &  8.5 & 11.0 & 39.1--50.4 & 30.2--38.9 \\
United States &  3.40 &  7.2 &  8.0 & 25.9--28.6 & 20.3--22.5 \\
Germany       &  1.51 &  5.5 &  7.4 & 57.6--77.9 & 45.3--61.4 \\
\hline
\end{tabular}
\end{table}

\paragraph{Survival per unit of stock.} While a node holds stock, its drawdown rate is its loss share times its required inflow, so its survival time is proportional to its stock. The last two columns of Table~\ref{tab:downstream_flow_comparison} use this to state results that do not depend on any stock assumption: the number of weeks for which 30 days of cover last. Other stock levels scale proportionally. At $z=0$, 30 days of cover last Japan 5.1 weeks and South Korea less than 7 weeks, while they last Germany more than a year.

\paragraph{Time to exhaustion.} Table~\ref{tab:buffer_debilitation_comparison} applies these rates to the initial stocks. At $z=0$, Japan's 183 days of cover are exhausted after 31.2 weeks and South Korea's 191 days after 43.4 weeks. Under the 90-day benchmark, India's stock would last 29.9 weeks and China's 37.7 weeks. Germany, Spain, France and the United States are not exhausted within the 52-week horizon at any safety level. At their loss rates the European stocks in the table would last roughly two years or longer. In the $z>0$ scenarios the requirement includes precautionary orders (Remark~\ref{rem:stock_flow}), which are filled from stock at the same loss share, so stocks are exhausted earlier. Moving from $z=0$ to $\bar z$ costs Japan 7.1 weeks and South Korea 7.3 weeks, and it reduces survival by 17\% to 25\% across the four Asian importers.

\begin{table}[ht]
\centering
\caption{Initial stocks and weeks to exhaustion after the closure. $^\dagger$Benchmark scenario of 90 days of net demand, not a measured stock.}\label{tab:buffer_debilitation_comparison}
\begin{tabular}{lrlrrrr}
\hline
 & Stock & & \multicolumn{3}{c}{Fixed price} & Endogenous price \\
Node & (MMbbl) & Basis & $z=0$ & $z=1.645$ & $z=\bar z$ & $z=\bar z$ \\
\hline
Japan         & 423 & IEA days            & 31.2 & 26.3 & 24.1 & 24.4 \\
South Korea   & 525 & IEA days            & 43.4 & 38.4 & 36.1 & 37.5 \\
India         & 475 & 90 days$^\dagger$   & 29.9 & 24.6 & 22.4 & 22.6 \\
China         & 999 & 90 days$^\dagger$   & 37.7 & 32.4 & 30.1 & 30.8 \\
Spain         & 129 & IEA days            & $>52$ & $>52$ & $>52$ & $>52$ \\
France        & 125 & IEA days            & $>52$ & $>52$ & $>52$ & $>52$ \\
United States & 306 & 90 days$^\dagger$   & $>52$ & $>52$ & $>52$ & $>52$ \\
Germany       & 184 & IEA days            & $>52$ & $>52$ & $>52$ & $>52$ \\
\hline
\end{tabular}
\end{table}

\begin{figure}[ht]
\centering
\includegraphics[width=\textwidth]{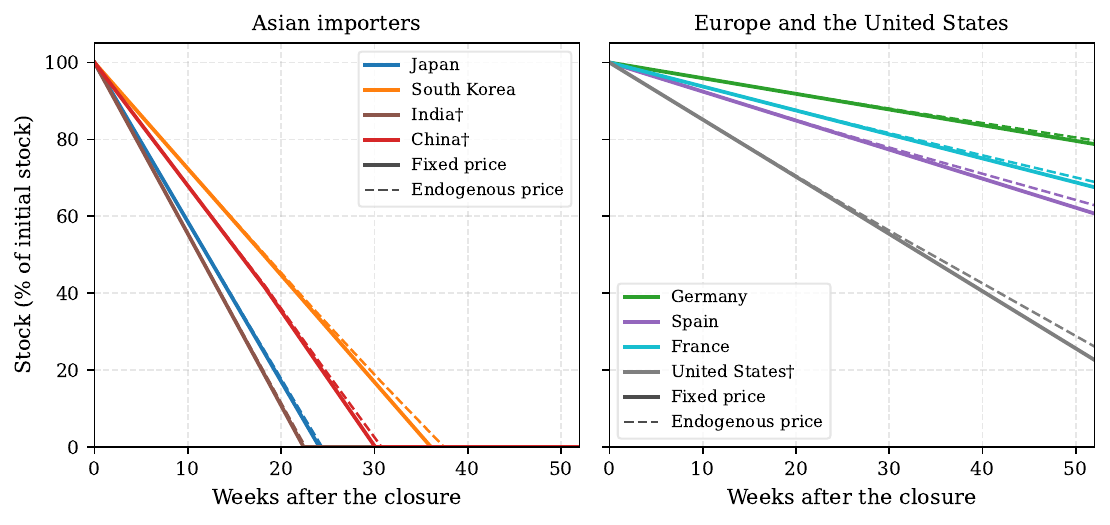}
\caption{Stocks of the eight importers after the closure at $z=\bar z$, in percent of the initial stock. Solid lines: fixed price. Dashed lines: endogenous price. The paths are piecewise linear, and the times at which they reach zero are the last two columns of Table~\ref{tab:buffer_debilitation_comparison}. $^\dagger$Initial stock is the 90-day benchmark.}\label{fig:exp_stocks}
\end{figure}

\paragraph{Price feedback.} In the endogenous-price variant at $\bar z$, higher prices reduce orders and exhaustion is delayed by 0.3 weeks for Japan, 1.4 weeks for South Korea, 0.7 weeks for China and 0.2 weeks for India.
 Figure~\ref{fig:exp_stocks} shows the stock paths at $\bar z$ under both price variants. The effect is small because demand is inelastic and the price is capped. We do not report price levels. The network contains only internationally traded crude, with 39.3 mb/d of net demand, which is well below world consumption. The proportional loss of supply, and with it the price response, is therefore overstated. The timing of the price increases also depends on the benchmark stocks of India and China.

\subsection{Order inflation under the closure}\label{sec:exp_game}

This subsection applies the shortage game of Section~\ref{sec:shortage} to the same network and the same closure: the four most exposed importers, Japan, South Korea, India and China, can order more than they need while the other nodes cannot.

\paragraph{Design.} The game is posed at $z=0$ for the long-run regime, in which exposed nodes hold no stock. Write $\lambda^0$ for the baseline requirements. Three choices differ from Section~\ref{sec:exp_design}. First, the transit nodes other than the Strait of Hormuz are not players: they pass on the orders they receive and deliver what they receive, so their customers order directly from the Hormuz node and requirements and truthful fill rates are unchanged. Second, the order cap of every node is $C_j=1.05\,\lambda^0_j$, without the historical-capacity term of Section~\ref{sec:exp_data}, and the cap of the Hormuz node is 10\% of its baseline requirement. Third, the cap of each of the four importers is raised to $(1+h)\lambda^0_j$, with $h\in\{5\%,10\%,25\%,50\%,100\%\}$. At $h=5\%$ all nodes have the same proportional headroom, and then no node other than the Hormuz node can have a requirement above its order cap: for every order profile $x\le C$ and every such node $i$, $\lambda_i(x)\le\alpha_i+1.05\sum_j p_{ij}\lambda^0_j\le 1.05\,\lambda^0_i=C_i$. Unequal headroom is therefore what allows a shortage to spread to suppliers with no physical shortage, as in Example~\ref{exm:bottleneck}. The values of $h$ are scenarios and are not estimated.

\paragraph{Equal headroom.} At $h=5\%$ every node can order 5\% more than its baseline requirement. By the argument above no strategic bottleneck arises, and no node that delivers in full under truthful ordering rations. Served demand rises from 25.08 to 26.22 mb/d: the four importers gain from 11.30 to 11.87 mb/d and all other nodes from 13.78 to 14.35 mb/d. The gain is idle capacity drawn into use, as in Theorem~\ref{thm:network}(b): exporters fill the inflated orders up to their caps. Its size is set by the 5\% headroom assumed for the exporters and is not an estimate of spare production capacity.

\paragraph{Unequal headroom.} When only the four importers can order more, total served demand hardly changes: 26.22 mb/d at $h=5\%$ and 26.32 mb/d at $h=100\%$. The allocation changes. The served demand of the four importers rises from 11.87 to 14.17 mb/d, and that of all other nodes falls from 14.35 to 12.15 mb/d. Relative to truthful ordering, the other nodes are slightly better off up to $h=25\%$ (13.83 against 13.78 mb/d) and worse off from $h=50\%$ (13.19 mb/d); at $h=100\%$ they are served 1.63 mb/d, or 12\%, less. The served share of each of the four Asian importers rises with $h$; that of each of the other four rises at $h=5\%$, falls thereafter, and at $h=100\%$ is below its level under truthful ordering, with the largest fall for the United States (Figure~\ref{fig:shortage_game}).

\paragraph{Strategic bottlenecks.} At $h=10\%$, 29 nodes other than the Hormuz node have requirements above their order caps in equilibrium, and 41 at $h=100\%$. The test of Theorem~\ref{thm:network}(e), which uses only the truthful clearing quantities, identifies 24 to 36 of them and flags no node that is not a bottleneck in equilibrium. Between 11 and 12 nodes that deliver in full under truthful ordering ration in equilibrium. The step from no bottleneck at $h=5\%$ to 29 at $h=10\%$ reflects that equal headroom is a boundary case: served demand at $h=5\%$ is within 0.10 mb/d of its level at $h=100\%$, so exporters operate close to their caps and any additional headroom of the four importers pushes them beyond.

\begin{table}[ht]
\centering
\caption{Shortage game under the closure at $z=0$: truthful ordering and the least shortage equilibrium for each order headroom $h$ of Japan, South Korea, India and China (the coalition). Served demand is $\sum_j\alpha_j\psi_j$. A strategic bottleneck is a node other than the Hormuz node whose requirement exceeds its order cap in equilibrium; ``flagged by test'' counts the nodes identified by the test of Theorem~\ref{thm:network}(e). ``Newly rationed'' counts nodes that deliver in full under truthful ordering and ration in equilibrium.}\label{tab:shortage_game}
{\small\setlength{\tabcolsep}{4pt}\begin{tabular}{lrrrrrrr}
\hline
 & \multicolumn{3}{c}{Served demand (mb/d)} & Strategic & Flagged & Newly & Sweeps \\
Ordering & total & coalition & others & bottlenecks & by test & rationed & (least / greatest) \\
\hline
Truthful & 25.08 & 11.30 & 13.78 & -- & -- & -- & -- \\
Equilibrium, $h=5\%$ & 26.22 & 11.87 & 14.35 & 0 & 0 & 0 & 790 / 835 \\
Equilibrium, $h=10\%$ & 26.27 & 12.04 & 14.23 & 29 & 24 & 11 & 804 / 801 \\
Equilibrium, $h=25\%$ & 26.30 & 12.47 & 13.83 & 35 & 32 & 12 & 93 / 36 \\
Equilibrium, $h=50\%$ & 26.31 & 13.12 & 13.19 & 39 & 34 & 12 & 45 / 35 \\
Equilibrium, $h=100\%$ & 26.32 & 14.17 & 12.15 & 41 & 36 & 12 & 33 / 24 \\
\hline
\end{tabular}
}
\end{table}

\FloatBarrier

The limitations of the experiment are discussed in Appendix~\ref{app:limits}.

\section{Conclusion}\label{sec:conclusion}
This paper formulates production in a capacitated supply chain network as a clearing problem. The clearing vector is the unique decentralized equilibrium of Newsvendor ordering, is computed by one linear program, and partitions the network into hedged and unhedged nodes. Its comparative statics are given by $M$-matrix multipliers; the stationary covariance of pass-through orders has a closed form; inventory drawdown after a capacity shock is a Skorokhod reflection problem with a closed-form, piecewise-linear solution and explicit times to exhaustion; a shortage game describes how rationed nodes inflate their orders; and an endogenous-price extension has a well-defined per-epoch equilibrium. Every layer is computed by one linear program or by closed-form expressions (Proposition \ref{prop:tractable}). The numerical section illustrates the model as a stress-testing instrument on global crude-oil trade.

Our analytical results give four insights for operations strategy and supply chain design:
\begin{itemize}
    \item \textbf{Static equilibrium.} The production profile of a capacitated network is a clearing vector: unique, decentralized, and computable by one linear program. Its dual prices are the network multipliers $(I_\calh - P_\calh)^{-1}$, which quantify how a unit of demand or capacity at one node propagates to every other node. Order targets compound along directed cycles and pool along tree paths, as the diagonal of the multiplier shows.
    \item \textbf{Exposure is structural; inventories buy time.} Capacity cascades run upstream in orders and downstream in deliveries. After a capacity loss, the set of nodes that are eventually rationed and their long-run fill rates are determined by the routing topology and the unhedged set alone (Proposition \ref{prop:exposure}). Strategic reserves change when a node runs out and never bring forward any node's exhaustion. Managers sizing buffers should therefore read them as survival time against a structural deficit, and should look to capacity or sourcing changes to remove exposure.
    \item \textbf{Shortage gaming redistributes and spreads rationing.} When rationed nodes can inflate their orders, every equilibrium has orders at least as large as truthful ones, and total deliveries are never lower than under truthful ordering: inflated orders draw idle capacity into use, so what changes is who is served. At a single bottleneck the shortage is allocated by customers' capacity rather than by need. In a network, an inflated order reaches every supplier of the inflating node, so a supplier with no physical shortage can become unhedged and ration customers that have no link to the original disruption; Theorem \ref{thm:network}(e) identifies such suppliers from primitives. Allocation rules, as well as capacities, therefore determine who bears a shortage. In the crude-oil experiment, order inflation raises total deliveries by about 1.2~mb/d relative to truthful ordering, and 1.1~mb/d of this gain is already obtained when every node has the same 5\% headroom. With more headroom, the four exposed importers raise their own served demand by up to 2.9~mb/d and lower that of all other nodes by up to 1.6~mb/d, both relative to truthful ordering, and as many as 41 nodes whose capacity was not shocked become unhedged (Section~\ref{sec:exp_game}).
    \item \textbf{Price feedback has two opposing effects.} Endogenous prices lower external demand and raise safety-stock targets. The first buys time; the second enlarges the unhedged set. Which dominates at a node is determined by the viability inequality of Proposition \ref{prop:price_equilibrium}.
\end{itemize}

We note three future research directions; Appendix~\ref{app:future} records them in detail.

\printbibliography

@article{arrow1951optimal,
  title     = {Optimal Inventory Policy},
  author    = {Arrow, Kenneth J. and Harris, Theodore and Marschak, Jacob},
  journal   = {Econometrica},
  volume    = {19},
  number    = {3},
  pages     = {250--272},
  year      = {1951}
}

@article{clark1960optimal,
  title     = {Optimal Policies for a Multi-Echelon Inventory Problem},
  author    = {Clark, Andrew J. and Scarf, Herbert},
  journal   = {Management Science},
  volume    = {6},
  number    = {4},
  pages     = {475--490},
  year      = {1960}
}

@article{lee1997bullwhip,
  title   = {The Bullwhip Effect in Supply Chains},
  author  = {Lee, Hau L. and Padmanabhan, V. and Whang, Seungjin},
  journal = {Sloan Management Review},
  volume  = {38},
  number  = {3},
  pages   = {93--102},
  year    = {1997}
}

@article{federgruen1984approximations,
  title     = {Approximations of Dynamic, Multilocation Production and Inventory Problems},
  author    = {Federgruen, Awi and Zipkin, Paul},
  journal   = {Management Science},
  volume    = {30},
  number    = {1},
  pages     = {69--84},
  year      = {1984}
}

@article{rosling1989optimal,
  title     = {Optimal Inventory Policies for Assembly Systems under Random Demands},
  author    = {Rosling, Kaj},
  journal   = {Operations Research},
  volume    = {37},
  number    = {4},
  pages     = {565--579},
  year      = {1989}
}

@article{langenhoff1990analytical,
  title     = {An Analytical Theory of Multi-Echelon Production/Distribution Systems},
  author    = {Langenhoff, L. J. G. and Zijm, W. H. M.},
  journal   = {Statistica Neerlandica},
  volume    = {44},
  number    = {3},
  pages     = {149--174},
  year      = {1990}
}

@article{graves2000optimizing,
  title     = {Optimizing Strategic Safety Stock Placement in Supply Chains},
  author    = {Graves, Stephen C. and Willems, Sean P.},
  journal   = {Manufacturing \& Service Operations Management},
  volume    = {2},
  number    = {1},
  pages     = {68--83},
  year      = {2000}
}

@incollection{graves2003supply,
  title     = {Supply Chain Design: Safety Stock Placement and Supply Chain Configuration},
  author    = {Graves, Stephen C. and Willems, Sean P.},
  booktitle = {Handbooks in Operations Research and Management Science},
  volume    = {11},
  publisher = {Elsevier},
  pages     = {95--132},
  year      = {2003}
}

@article{chen2000quantifying,
  author  = {Chen, Frank and Drezner, Zvi and Ryan, Jennifer K. and Simchi-Levi, David},
  title   = {Quantifying the Bullwhip Effect in a Simple Supply Chain: The Impact of Forecasting, Lead Times, and Information},
  journal = {Management Science},
  volume  = {46},
  number  = {3},
  pages   = {436--443},
  year    = {2000}
}

@article{ouyang2006characterization,
  title   = {Characterization of the Bullwhip Effect in Linear, Time-Invariant Supply Chains: Some Formulae and Tests},
  author  = {Ouyang, Yanfeng and Daganzo, Carlos},
  journal = {Management Science},
  volume  = {52},
  number  = {10},
  pages   = {1544--1556},
  year    = {2006}
}

@article{ouyang2010bullwhip,
  title   = {The Bullwhip Effect in Supply Chain Networks},
  author  = {Ouyang, Yanfeng and Li, Xiaopeng},
  journal = {European Journal of Operational Research},
  volume  = {201},
  number  = {3},
  pages   = {799--810},
  year    = {2010}
}

@article{osadchiy2021bullwhip,
  title   = {The Bullwhip Effect in Supply Networks},
  author  = {Osadchiy, Nikolay and Schmidt, William and Wu, Jing},
  journal = {Management Science},
  volume  = {67},
  number  = {10},
  pages   = {6153--6173},
  year    = {2021}
}

@book{chen2001fundamentals,
  title     = {Fundamentals of Queueing Networks: Performance, Asymptotics, and Optimization},
  author    = {Chen, Hong and Yao, David D.},
  series    = {Stochastic Modelling and Applied Probability},
  volume    = {46},
  publisher = {Springer},
  location  = {New York},
  year      = {2001}
}

@article{eisenberg2001systemic,
  title     = {Systemic Risk in Financial Systems},
  author    = {Eisenberg, Larry and Noe, Thomas H.},
  journal   = {Management Science},
  volume    = {47},
  number    = {2},
  pages     = {236--249},
  year      = {2001}
}

@article{dong2005multitiered,
  title     = {Multitiered Supply Chain Networks: Multicriteria Decision-Making Under Uncertainty},
  author    = {Dong, June and Zhang, Ding and Yan, Hong and Nagurney, Anna},
  journal   = {Annals of Operations Research},
  volume    = {135},
  number    = {1},
  pages     = {155--178},
  year      = {2005}
}

@article{cachon2007search,
  title     = {In Search of the Bullwhip Effect},
  author    = {Cachon, G{\'e}rard P. and Randall, Taylor and Schmidt, Glen M.},
  journal   = {Manufacturing \& Service Operations Management},
  volume    = {9},
  number    = {4},
  pages     = {457--479},
  year      = {2007}
}

@article{cifuentes2005liquidity,
  title     = {Liquidity Risk and Contagion},
  author    = {Cifuentes, Rodrigo and Ferrucci, Gianluigi and Shin, Hyun Song},
  journal   = {Journal of the European Economic Association},
  volume    = {3},
  number    = {2--3},
  pages     = {556--566},
  year      = {2005}
}

@article{caccioli2014stability,
  title     = {Stability Analysis of Financial Contagion due to Overlapping Portfolios},
  author    = {Caccioli, Fabio and Shrestha, Munik and Moore, Cristopher and Farmer, J. Doyne},
  journal   = {Journal of Banking \& Finance},
  volume    = {46},
  pages     = {233--245},
  year      = {2014}
}

@article{glasserman2016contagion,
  title     = {Contagion in Financial Networks},
  author    = {Glasserman, Paul and Young, H. Peyton},
  journal   = {Journal of Economic Literature},
  volume    = {54},
  number    = {3},
  pages     = {779--831},
  year      = {2016}
}

@article{veraart2020distress,
  title     = {Distress and Default Contagion in Financial Networks},
  author    = {Veraart, Luitgard Anna Maria},
  journal   = {Mathematical Finance},
  volume    = {30},
  number    = {3},
  pages     = {705--737},
  year      = {2020}
}

@article{simchilevi2015risk,
  title={Identifying Risks and Mitigating Disruptions in the Automotive Supply Chain},
  author={Simchi-Levi, David and Schmidt, William and Wei, Yehua and Zhang, Peter Yun and Combs, Keith and Ge, Yao and Gusikhin, Oleg and Sanders, Michael and Zhang, Don},
  journal={Interfaces},
  volume={45},
  number={5},
  pages={375--390},
  year={2015},
  publisher={INFORMS}
}

@article{acemoglu2012network,
  title={The network origins of aggregate fluctuations},
  author={Acemoglu, Daron and Carvalho, Vasco M and Ozdaglar, Asuman and Tahbaz-Salehi, Alireza},
  journal={Econometrica},
  volume={80},
  number={5},
  pages={1977--2016},
  year={2012},
  publisher={Wiley Online Library}
}

@article{carvalho2021supply,
  title={Supply chain disruptions: Evidence from the Great East Japan Earthquake},
  author={Carvalho, Vasco M and Nirei, Makoto and Saito, Yukiko U and Tahbaz-Salehi, Alireza},
  journal={The Quarterly Journal of Economics},
  volume={136},
  number={2},
  pages={1255--1321},
  year={2021},
  publisher={Oxford University Press}
}

@article{ang2017disruption,
  title={Disruption risk and optimal sourcing in multitier supply networks},
  author={Ang, Erjie and Iancu, Dan A and Swinney, Robert},
  journal={Management Science},
  volume={63},
  number={8},
  pages={2397--2419},
  year={2017},
  publisher={INFORMS}
}

@article{osadchiy2016systematic,
  title={Systematic risk in supply chain networks},
  author={Osadchiy, Nikolay and Gaur, Vishal and Seshadri, Sridhar},
  journal={Management Science},
  volume={62},
  number={6},
  pages={1755--1777},
  year={2016},
  publisher={INFORMS}
}

@article{bimpikis2019supply,
  title={Supply disruptions and optimal network structures},
  author={Bimpikis, Kostas and Candogan, Ozan and Ehsani, Shayan},
  journal={Management Science},
  volume={65},
  number={12},
  pages={5504--5517},
  year={2019},
  publisher={INFORMS}
}

@article{skorokhod1961stochastic,
  author  = {Skorokhod, Anatoliy V.},
  title   = {Stochastic Equations for Diffusion Processes in a Bounded Region},
  journal = {Theory of Probability \& Its Applications},
  volume  = {6}, number = {3}, pages = {264--274}, year = {1961}
}

@article{harrison1981reflected,
  author  = {Harrison, J. Michael and Reiman, Martin I.},
  title   = {Reflected {B}rownian Motion on an Orthant},
  journal = {The Annals of Probability},
  volume  = {9}, number = {2}, pages = {302--308}, year = {1981}
}

@book{harrison2013brownian,
  author    = {Harrison, J. Michael},
  title     = {Brownian Models of Performance and Control},
  publisher = {Cambridge University Press},
  address   = {Cambridge},
  year      = {2013}
}

@article{chen2021financial,
  author  = {Chen, Hong and Wang, Tan and Yao, David D.},
  title   = {Financial Network and Systemic Risk---A Dynamic Model},
  journal = {Production and Operations Management},
  volume  = {30}, number = {8}, pages = {2441--2466}, year = {2021}
}

@article{cachon1999capacity,
  author  = {Cachon, G{\'e}rard P. and Lariviere, Martin A.},
  title   = {Capacity Choice and Allocation: Strategic Behavior and Supply Chain Performance},
  journal = {Management Science},
  volume  = {45}, number = {8}, pages = {1091--1108}, year = {1999}
}

@article{tarski1955lattice,
  author  = {Tarski, Alfred},
  title   = {A Lattice-Theoretical Fixpoint Theorem and Its Applications},
  journal = {Pacific Journal of Mathematics},
  volume  = {5}, number = {2}, pages = {285--309}, year = {1955}
}

\newpage
\appendix
\section{Proofs for Section \ref{sec:model}}
This appendix collects the proofs of the results in Section \ref{sec:model}.

\subsection{Proof of Theorem \ref{thm:static_solution}}\label{app:proof_thm1}
\begin{proof}
(a) For $q, q' \in [0, C]$, $|T(q) - T(q')| \le |(\alpha + Pq) - (\alpha + Pq')| = P|q - q'|$ componentwise, using $|a \wedge c - b \wedge c| \le |a - b|$ and $P \ge 0$. Hence $\norm{T(q) - T(q')}_u \le \norm{P|q - q'|}_u \le \gamma \norm{q - q'}_u$ since $Pu \le \gamma u$. $T$ maps $[0, C]$ into itself because $\alpha \ge 0$. Banach's theorem gives a unique fixed point $q^*$ and convergence of $T^n$ from any starting point.

(b) $\ex[\pi_i]$ is concave in $q_i$ with derivative $p_i \pr(D_i > q_i) - c_i$, which vanishes at $q_i = \mu_i(q) + z_i \sigma_i$. The best response on $[0, C_i]$ is therefore \eqref{eq:best_response}. A Nash equilibrium is a fixed point of $T$ and conversely; uniqueness follows from (a).

(c) $q^*$ is feasible for \eqref{eq:lp}: $q^* \le C$ and $q^* \le \alpha + Pq^*$. Let $q$ be any feasible point. Then $q \le (\alpha + Pq) \wedge C = T(q)$; $T$ is monotone, so $q \le T(q) \le T^2(q) \le \dots \to q^*$. Thus $q^* \ge q$ for every feasible $q$, so $q^*$ maximizes $\ind^\top q$, and any other maximizer $q'$ satisfies $q' \le q^*$ and $\ind^\top q' = \ind^\top q^*$, forcing $q' = q^*$.

(d) For $i \in U$, $q^*_i = C_i$ by \eqref{eq:best_response}. For $i \in \calh$, $q^*_i = \alpha_i + (Pq^*)_i$; substituting $q^*_U = C_U$ gives $(I_\calh - P_\calh) q^*_\calh = \alpha_\calh + P_{\calh,U} C_U$, and $I_\calh - P_\calh$ is a nonsingular $M$-matrix because $P_\calh$ is a principal submatrix of $P$ with $\rho(P_\calh) \le \rho(P) < 1$. The dual of \eqref{eq:lp} is $\min y\alpha + wC$ s.t.\ $y(I - P) + w \ge \ind^\top$, $y, w \ge 0$. The stated $(y, w)$ with $y_U = 0$, $w_\calh = 0$ is feasible: on $\calh$, $y_\calh (I_\calh - P_\calh) = \ind^\top$; on $U$, $-y_\calh P_{\calh,U} + w_U = \ind^\top$. It satisfies complementary slackness with $q^*$ (demand constraints bind on $\calh$, capacity constraints bind on $U$), hence is optimal; $y_\calh \ge \ind^\top$ because $(I_\calh - P_\calh)^{-1} \ge I_\calh$.
\end{proof}

\subsection{Proof of Proposition \ref{prop:computation}}\label{app:proof_prop2}
\begin{proof}
(a) $q^{(0)} = (I - P)^{-1}\alpha$ satisfies $T(q^{(0)}) \le \alpha + Pq^{(0)} = q^{(0)}$; monotonicity of $T$ gives a nonincreasing sequence bounded below by $0$, converging to the unique fixed point at rate $\gamma$ by Theorem \ref{thm:static_solution}(a). (b) $T$ is monotone in $(\alpha, C)$: if $\alpha' \ge \alpha$ and $C' \ge C$ then $T'(q^*) \ge T(q^*) = q^*$, so iterating $T'$ from $q^*$ yields a nondecreasing sequence converging to $q^{*\prime} \ge q^*$. (c) If no constraint is tied, small perturbations of $(\alpha, C)$ leave the active set of \eqref{eq:lp} unchanged, and \eqref{eq:partition} is affine on that set.
\end{proof}

\subsection{Proof of Theorem \ref{thm:variance}}\label{app:proof_thm3}
\begin{proof}
Subtracting means in \eqref{eq:passthrough} and taking covariances, independence of $\tilde D^{\mathrm{ext}}(s)$ from $\tilde D(s-1)$ gives $\Sigma = P \Sigma P^\intercal + \Sigma_D$. Since $\rho(P) < 1$, $\sum_k P^k \Sigma_D (P^\intercal)^k$ converges and solves the equation; the solution is unique because the operator $X \mapsto X - PXP^\intercal$ has spectrum $\{1 - \lambda_i \lambda_j\}$ bounded away from zero. For diagonal $\Sigma_D$, $\Sigma_{ii} = \sum_k \sum_m ([P^k]_{im})^2 w_m^2 = [N_\sigma w^2]_i$, which is \eqref{eq:covariance_closed}; $N_\sigma$ converges because $([P^k]_{im})^2 \le [P^k]_{im}$ and $\sum_k P^k$ converges. For the comparison, $([P^k]_{ji})^2 = (\sum_{\text{paths}} \prod p)^2 \ge \sum_{\text{paths}} \prod p^2 = [(P \odot P)^k]_{ji}$, with equality iff at most one path of length $k$ from $i$ to $j$ has positive weight. Summing gives $N_\sigma \ge M_\sigma$.
\end{proof}

\subsection{Proof of Proposition \ref{prop:sens}}\label{app:proof_prop4}
\begin{proof}
\eqref{eq:sens_v}--\eqref{eq:sens_C} follow by differentiating \eqref{eq:partition} under local constancy of the partition (Proposition \ref{prop:computation}(c)). For \eqref{eq:sens_sigma}, differentiate \eqref{eq:lyapunov}: $\partial\Sigma/\partial w_i = \sum_k P^k (2 w_i e_i e_i^\intercal) (P^\intercal)^k$, so $\partial\sigma_j^2/\partial w_i = 2 w_i \sum_k ([P^k]_{ji})^2 = 2 w_i [N_\sigma]_{ji}$ and $\partial\sigma_j/\partial w_i = w_i [N_\sigma]_{ji}/\sigma_j$. For \eqref{eq:sens_w}, $\partial\alpha_\calh/\partial w_i = \operatorname{diag}(z_\calh)\, \partial\sigma_\calh/\partial w_i$ and \eqref{eq:sens_v} applies. Nonnegativity follows from $N_\sigma \ge 0$, $(I_\calh - P_\calh)^{-1} \ge 0$, and $z > 0$ (Assumption \ref{defn:hedging}).
\end{proof}

\section{Proofs for Section \ref{sec:dynamic}}
The arguments follow the fluid-network treatment of \textcite{chen2001fundamentals} and the dynamic clearing model of \textcite{chen2021financial}.

\subsection{Proof of Theorem \ref{thm:drawdown}}\label{app:proof_thm5}
\begin{proof}
(a) Let $\Lambda := \operatorname{diag}(\lambda)$, $D := \operatorname{diag}(q)$ and $K := D\Lambda^{-1} = \operatorname{diag}(q_i/\lambda_i)$, so $0 \le K \le I$. Then $\tilde P = \Lambda^{-1} P D$ and $\Lambda \tilde P \Lambda^{-1} = PK$. Since $0 \le PK \le P$ entrywise, $\rho(\tilde P) = \rho(PK) \le \rho(P) < 1$. The reflection matrix $I - \tilde P^\intercal$ is therefore of the Harrison--Reiman type, and existence, uniqueness and Lipschitz continuity follow from \textcite{harrison1981reflected}. For later use we recall the construction: for fixed $Y$, coordinate $i$ of \eqref{eq:skorokhod} is a one-dimensional Skorokhod problem for $X_i - (\tilde P^\intercal Y)_i$, so solutions of \eqref{eq:skorokhod} are exactly the fixed points of
\[
  \Gamma_X(Y)(t) := \sup_{0 \le r \le t} \big[ \tilde P^\intercal Y(r) - X(r) \big]^+ \quad \text{(componentwise)},
\]
and $\Gamma_X$ is a contraction in the norm $\sum_i u_i \sup_{t \le T} |Y_i(t)|$, where $u := (I - \tilde P/\gamma)^{-1} \ind$ for some $\gamma \in (\rho(\tilde P), 1)$, so that $\tilde P u \le \gamma u$.

(b)--(c) We construct a piecewise-linear solution; by (a) it is the solution. Let $E_0 := \{i : Z_i(0) = 0\}$ and define $L$ by \eqref{eq:shortfall} with $E = E_0$. Then $L \ge 0$ because $\lambda - q \ge 0$ and $(I_E - \tilde P_{EE}^\intercal)^{-1} \ge 0$. With $\dot Y = L$ and $\dot Z = \theta^0 + (I - \tilde P^\intercal) L$, nodes in $E_0$ have $\dot Z = 0$ by \eqref{eq:shortfall} and stay at zero, and a stocked node $j$ has $\dot Z_j = \theta^0_j - \sum_{i \in E} \tilde P_{ij} L_i = \theta_j \le 0$. All conditions of \eqref{eq:skorokhod} hold until the first time $\tau$ of \eqref{eq:tau} at which some stocked nodes reach zero. Add them to $E$ and repeat: with the new rates every node of the enlarged $E$ again has $\dot Z = 0$. Since $E$ strictly increases at each epoch there are at most $N$ epochs, and $Z$ is nonincreasing. Monotonicity of $L$: $L_E = \sum_{n \ge 0} (\tilde P_{EE}^\intercal)^n (\lambda - q)_E$ is a sum of nonnegative terms over paths inside $E$, hence nondecreasing in $E$. The bound $L \le \lambda$: iterate $L^{(n+1)}_E = (\lambda - q)_E + \tilde P_{EE}^\intercal L^{(n)}_E$ from $L^{(0)} = 0$; if $L^{(n)} \le \lambda$ then $(\tilde P^\intercal L^{(n)})_k \le (\tilde P^\intercal \lambda)_k = \varsigma_k q_k \le q_k$, so $L^{(n+1)}_k \le \lambda_k$.

(d) After the last epoch every stocked node has $\dot Z = 0$, since a stocked node with negative drift would reach zero in finite time. Hence $L$, extended by zero off $E$, satisfies $L = (\lambda - q) + \tilde P^\intercal L$ on the full node set: on $E$ this is \eqref{eq:shortfall}, and for a stocked node $k$ it reads $0 = (\lambda - q)_k + (\tilde P^\intercal L)_k$, that is, $\dot Z_k = 0$. The unique solution of that system is \eqref{eq:Linf}.

(e) If $U = \emptyset$ then $\theta^0 = \bm 0$ and $(Z, Y) = (Z(0), \bm 0)$ satisfies \eqref{eq:skorokhod}.
\end{proof}

\subsection{Proof of Proposition \ref{prop:exposure}}\label{app:proof_exposure}
\begin{proof}
(i) $L^\infty = \sum_{n \ge 0} (\tilde P^\intercal)^n (\lambda - q)$, and $(\lambda - q)_j > 0$ exactly for $j \in U$, so $L^\infty_i > 0$ if and only if $[\tilde P^n]_{ji} > 0$ for some $j \in U$ and $n \ge 0$; and $\tilde P_{ab} > 0$ if and only if $p_{ab} q_b > 0$. By the path argument in the proof of Theorem \ref{thm:drawdown}, $L(t) \le L^\infty$ for all $t$. If $L^\infty_i = 0$ and $i$ is stocked, $\dot Z_i = \theta^0_i - (\tilde P^\intercal L(t))_i \ge \theta^0_i - (\tilde P^\intercal L^\infty)_i = -L^\infty_i = 0$, so $i$ never draws on its inventory. If $L^\infty_i > 0$, then by Theorem \ref{thm:drawdown}(d) $L_i = L^\infty_i > 0$ after the last epoch, which requires $Z_i = 0$.

(ii) From \eqref{eq:Linf}, $L^\infty = (\lambda - q) + \tilde P^\intercal L^\infty$, so $s^\infty = q - \tilde P^\intercal L^\infty = q - \tilde P^\intercal \lambda + \tilde P^\intercal s^\infty$, and $(\tilde P^\intercal \lambda)_k = \sum_j p_{jk} q_k = \varsigma_k q_k$. Nonnegativity is $L^\infty \le \lambda$. Equation \eqref{eq:Linf} does not involve $Z(0)$.

(iii) Let $Z(0) \le Z'(0)$, so $X \le X'$ pointwise. $\Gamma_X(Y)$ is nondecreasing in $Y$ and nonincreasing in $X$, so the iterates of $\Gamma_X$ and $\Gamma_{X'}$ from zero are ordered, and so are their limits: $Y' \le Y$. Fix $i$ with $Z_i(0) > 0$ and $t < \tau^{\mathrm{ex}}_i$; then $Y_i(t) = 0$ and
\[
  Z'_i(t) = X'_i(t) + Y'_i(t) - (\tilde P^\intercal Y'(t))_i \ge X_i(t) - (\tilde P^\intercal Y(t))_i = Z_i(t) > 0,
\]
so the exhaustion time of node $i$ from $Z'(0)$ is at least $\tau^{\mathrm{ex}}_i$. If $Z_i(0) = 0$ then $\tau^{\mathrm{ex}}_i = 0$ and there is nothing to prove.
\end{proof}

\subsection{Proof of Proposition \ref{prop:tractable}}\label{app:proof_tractable}
\begin{proof}
(i) This is Theorem \ref{thm:static_solution}(a) and (c). The matrix $I - P$ has at most $m + N$ nonzero entries, and a sweep of $T$ requires the product $Pq$ and $N$ minima.

(ii) By Theorem \ref{thm:drawdown}(b)--(c) the solution of \eqref{eq:skorokhod} is piecewise linear, with rates \eqref{eq:shortfall}--\eqref{eq:drawdown} between epochs and epochs given by \eqref{eq:tau}. Lines \ref{alg1:partition}--\ref{alg1:unhedged} compute the requirements $\lambda$ of \eqref{eq:flow}, line \ref{alg1:rates} evaluates \eqref{eq:shortfall}--\eqref{eq:drawdown}, and the loop containing line \ref{alg1:tau} evaluates \eqref{eq:tau}, so the algorithm reproduces the solution exactly on each inter-epoch interval. If a pass does not terminate at line \ref{alg1:break}, then $\tau$ is finite and $t + \tau < T$, so the nodes attaining the minimum in line \ref{alg1:tau} reach zero inventory. A node with zero inventory has $\theta_j = 0$ and stays at zero. Hence the stocked set loses at least one node in every such pass, there are at most $N$ of them, and one further pass terminates. The horizon $T$ enters only through the cap on the final step and the stopping test.

(iii) Equation \eqref{eq:Linf} is established in Theorem \ref{thm:drawdown}(d). By Proposition \ref{prop:exposure}(i) the exposed set is the set of nodes reachable from $U$ along links $a \to b$ with $p_{ab} q_b > 0$.

(iv) In Algorithm \ref{alg:enhanced_strategic_dynamics} the drawdown rates are those of line \ref{alg1:rates} of Algorithm \ref{alg:sequential_dynamics} at the current data, so inventories are nonincreasing and empty nodes stay empty; the counting argument of (ii) applies.
\end{proof}

\section{Proofs for Section \ref{sec:pricing}}
\subsection{Proof of Lemma \ref{lem:rebalancing}}\label{app:proof_lemma6}

\begin{proof}
The proof shows that rule \eqref{eq:rebalancing} preserves column masses and, under the reachability condition of the lemma, the spectral bound.
Let $P(t) = [p_{ij}(t)] \in \mathbb{R}^{N \times N}$ represent the dynamic routing matrix, and let $P(0)$ be the initial, nominal routing matrix. By assumption, $P(0) \ge \bm{0}$ is substochastic with column masses $\varsigma_j := \sum_{m=1}^N p_{mj}(0) \le 1$ and spectral radius $\rho(P(0)) < 1$; write $\delta_j := 1 - \varsigma_j \ge 0$ for the external leakage of column $j$. Openness of the network means $\rho(P(0))<1$, which holds whenever every column has access, through the adjacency structure of $P(0)^\intercal$, to a \emph{deficient} column ($\varsigma_j<1$), e.g., a pure-supply node whose column is zero.

At any failure epoch $\tau_k$ (where $k \ge 1$), the set of active nodes, those with zero shortfall rate, is
\begin{equation}
    \cala_k = \{j \in \{1,\dots,N\} \mid L_j(\tau_k) = 0\}.
\end{equation}
The rebalancing rule \eqref{eq:rebalancing} updates each element as
\begin{equation}
    p_{ij}(\tau_k) = \ind\{i \in \cala_k\}\,\varsigma_j\,
    \frac{p_{ij}(\tau_{k-1})}{\sum_{m \in \cala_k} p_{mj}(\tau_{k-1})}.
\end{equation}
Summing over rows,
\begin{equation}
    \sum_{i=1}^N p_{ij}(\tau_k)
    = \varsigma_j\,\frac{\sum_{i\in\cala_k}p_{ij}(\tau_{k-1})}{\sum_{m\in\cala_k}p_{mj}(\tau_{k-1})}
    = \varsigma_j.
\end{equation}
Hence the column masses of $P(\tau_k)$ are invariant across epochs: every column retains its initial mass $\varsigma_j = 1-\delta_j$, so $P(\tau_k)$ remains substochastic with the same leakage profile as $P(0)$.

By construction of the indicator function in \eqref{eq:rebalancing}, we have $p_{ij}(\tau_k) = 0$ for all $i \in \cala_k^c$ and all $j$, where $\cala_k^c = \{1, \dots, N\} \setminus \cala_k$. Thus, the routing matrix $P(\tau_k)$ has the block structure:
\begin{equation}
    P(\tau_k) = \begin{pmatrix} 
        P_{\cala_k, \cala_k}(\tau_k) & P_{\cala_k, \cala_k^c}(\tau_k) \\ 
        \mathbf{0} & \mathbf{0} 
    \end{pmatrix}.
\end{equation}
The spectrum of $P(\tau_k)$, denoted $\text{spec}(P(\tau_k))$, consists of the eigenvalues of the blocks on the diagonal:
\begin{equation}
    \text{spec}(P(\tau_k)) = \text{spec}(P_{\cala_k, \cala_k}(\tau_k)) \cup \{0\}.
\end{equation}

For the strict spectral bound, recall the standard characterization for nonnegative substochastic matrices: $\rho(B)<1$ if and only if from every index there is a path, in the adjacency structure of $B^\intercal$, to a deficient column, i.e., one with column sum strictly below one. The rebalancing rule preserves the column sums $\{\varsigma_j\}$ exactly, but it can remove the paths to deficient columns: if two nodes each source from one another and from an outside supplier, and both outside suppliers ration, rebalancing leaves the two nodes sourcing entirely from each other and the spectral radius equals one. This is excluded by the reachability condition of the lemma: every index of $\cala_k$ reaches a deficient column within $P_{\cala_k,\cala_k}(\tau_k)$, whose column sums are again $\varsigma_j$ because the rows outside $\cala_k$ vanish. Consequently
\begin{equation}
    \rho(P_{\cala_k, \cala_k}(\tau_k)) < 1,
\end{equation}
and $\rho(P(\tau_k)) < 1$ for all epochs $\tau_k$.

We now define the matrix $A(t) = I - P(t)$ for $t \in [\tau_k, \tau_{k+1})$. Since $P(t)$ is constant between epochs:
\begin{enumerate}
    \item The off-diagonal entries $[A(t)]_{ij} = -p_{ij}(t) \le 0$ are non-positive since $p_{ij}(t) \ge 0$. Hence, $A(t)$ is a $Z$-matrix.
    \item Since $\rho(P(t)) < 1$, the matrix $A(t) = I - P(t)$ is a non-singular $M$-matrix. Its inverse exists and is element-wise non-negative:
    \begin{equation}\label{eq:mpoints_inverse}
        (I - P(t))^{-1} = \sum_{n=0}^{\infty} P(t)^n \ge \mathbf{0}.
    \end{equation}
\end{enumerate}
\end{proof}

\subsection{Proof of Proposition \ref{prop:price_equilibrium}}\label{app:proof_prop7}
\begin{proof}
(a) $\alpha(p)$ is continuous on $[p_0/2, \kappa p_0]$. The clearing vector is Lipschitz in $\alpha$, being the fixed point of a contraction whose modulus does not depend on $\alpha$ (Theorem \ref{thm:static_solution}(a)). Hence $\lambda$, $\tilde P$ and $L_E = (I_E - \tilde P_{EE}^\intercal)^{-1} (\lambda - \lambda \wedge C)_E$ are continuous in $p$ where $\lambda > 0$, and so is $Q(p)$. Thus $F$ is a continuous map of a compact interval into itself (as $Q(p) \to 0$ the projected value tends to $\kappa p_0$), and it has a fixed point.

(b) Differentiating $v_i(p) = v_i(0)(p_0/p)^{\varepsilon}$ and $z_i(p) = \Phi^{-1}(1 - c_i/p_i(p))$ with $p_i(p) = p_i(0)p/p_0$ gives
\begin{equation}
    \frac{\partial \alpha_i}{\partial p} = -\frac{\varepsilon v_i(p)}{p} + \frac{\sigma_i c_i}{p_i(p)\, p\, \phi(z_i(p))},
\end{equation}
which is nonpositive if and only if $\varepsilon\, v_i(p)\, \phi(z_i(p))\, p_i(p) \ge c_i \sigma_i$.

(c) $F(p) \ge p$ if and only if $h(p) \le p_0^{\varepsilon+\eta} Q(0)$. If $h$ is strictly increasing, $F(p) - p$ changes sign at most once, from positive to negative, so the projected map has exactly one fixed point (the interior root, or an endpoint). If $E \cap \cals = \emptyset$ then $L_i = 0$ for $i \in \cals$ and
\[
  h(p) = \sum_{i \in \cals} \max\big\{ 0,\ v_i(0)\, p_0^{\varepsilon} p^{\eta} + z(p)\, \sigma_i\, p^{\varepsilon+\eta} \big\}.
\]
With $v_i(0) \ge 0$ and $\Phi(\bar z) > 3/4$, $z(p)$ is positive and increasing on $[p_0/2, \kappa p_0]$, so every term is nondecreasing, and at least one term is strictly increasing under either of the stated conditions.

\emph{Arbitrary empty set.} Fix $E$ and let $\mathcal{A}$ be its complement. For a target vector $a\ge 0$ with $a_i=0$ for $i\notin\mathcal{S}$ and with requirements $\lambda(a)>0$, write $q(a)$ and $\lambda(a)$ for the clearing quantities, $\xi(a):=q(a)/\lambda(a)=\min\{1,C/\lambda(a)\}$, $\psi(a):=1-L/\lambda(a)$ for the fill rates at the shortfall rates $L$ of Section~\ref{sec:sequential}, and $\bar Q(a):=\sum_k a_k\psi_k(a)$, so that $Q(p)=\bar Q(\alpha(p))$. We drop the argument $a$ where no confusion arises, and $\Psi(\psi)=\mathbf{1}-\varsigma+P^{\intercal}\psi$ is as in Section~\ref{sec:shortage}.

\emph{Step 1: fill rates fall with targets.} The shortfall rates satisfy $L_k=(\lambda-q)_k+\sum_{i\in E}\tilde P_{ik}L_i$ for $k\in E$ and $L_k=0$ for $k\in\mathcal{A}$. Dividing by $\lambda_k$ gives $\psi_k=\xi_k\Psi_k(\psi)$ for $k\in E$ and $\psi_k=1$ for $k\in\mathcal{A}$. Hence $\psi_E=\sum_{n\ge 0}\bigl(\operatorname{diag}(\xi_E)P_{EE}^{\intercal}\bigr)^n(\xi_E\odot c_E)$ with $c_k:=1-\varsigma_k+\sum_{i\in\mathcal{A}}p_{ik}\ge 0$, which is nondecreasing in $\xi$. By Proposition~\ref{prop:computation}(b), $q$ and $\lambda$ are nondecreasing in $a$, so $\xi$ is nonincreasing in $a$, and therefore $\psi(a)$ is nonincreasing in $a$.

\emph{Step 2: conservation.} An empty node delivers what it receives, $\psi_k\lambda_k=q_k\Psi_k(\psi)$, and a stocked node delivers $\lambda_k=q_k\Psi_k(\psi)+d_k$ with drawdown $d_k:=\lambda_k-q_k\Psi_k(\psi)\ge 0$. Summing over $k$ and using $\lambda_k=a_k+\sum_j p_{kj}q_j$,
\[
\sum_k a_k\psi_k+\sum_k\sum_j\psi_k\,p_{kj}q_j=\sum_k(1-\varsigma_k)q_k+\sum_k\sum_i\psi_i\,p_{ik}q_k+\sum_{k\in\mathcal{A}}d_k .
\]
The two double sums coincide, so
\[
\bar Q(a)=\sum_k(1-\varsigma_k)\,q_k+\sum_{k\in\mathcal{A}}\bigl[\lambda_k-(\lambda_k\wedge C_k)\Psi_k(\psi)\bigr]:
\]
served demand equals own production plus the drawdown of stocked nodes.

\emph{Step 3: $\bar Q$ is nondecreasing in $a$.} The first sum is nondecreasing because $q$ is. Each bracket is nondecreasing in $\lambda_k$, because $\Psi_k\le 1$, and nonincreasing in $\Psi_k$; $\lambda$ is nondecreasing in $a$, and $\Psi(\psi(a))$ is nonincreasing in $a$ by Step~1.

\emph{Step 4: scaling.} Let $\theta\in(0,1]$. Since $C\ge\theta C$, $(\theta a+P\theta q(a))\wedge C\ge\theta\,[(a+Pq(a))\wedge C]=\theta q(a)$, so iterating the monotone clearing map with targets $\theta a$ from $\theta q(a)$ gives $q(\theta a)\ge\theta q(a)$ and $\lambda(\theta a)\ge\theta\lambda(a)>0$. By Step~1, $\psi(\theta a)\ge\psi(a)$, hence $\bar Q(\theta a)=\theta\sum_k a_k\psi_k(\theta a)\ge\theta\,\bar Q(a)$.

\emph{Step 5: conclusion.} For $i\in\mathcal{S}$, $p^{\varepsilon}\alpha_i(p)=v_i(0)p_0^{\varepsilon}+z(p)\sigma_i p^{\varepsilon}$, which is nondecreasing on $[p_0/2,\kappa p_0]$ because $z(p)$ is positive and increasing there. So for $p<p'$ and $\theta:=(p/p')^{\varepsilon}$ we have $\alpha(p')\ge\theta\alpha(p)$, and Steps~3--4 give $Q(p')\ge\bar Q(\theta\alpha(p))\ge\theta Q(p)$: the function $p^{\varepsilon}Q(p)$ is nondecreasing. If $\eta>0$, then $h(p)=p^{\eta}\cdot p^{\varepsilon}Q(p)$ is strictly increasing wherever $Q(p)>0$. If $\eta=0$ and $\sigma_i>0$ for every $i\in\mathcal{S}$ with $v_i(0)>0$, then $p^{\varepsilon}\alpha_i(p)$ is strictly increasing for every $i$ with $\alpha_i\not\equiv 0$, so $\alpha(p')\ge\theta'\alpha(p)$ for some $\theta'\in(\theta,1]$ and $Q(p')\ge\theta'Q(p)>\theta Q(p)$ whenever $Q(p)>0$. In both cases $h$ is nondecreasing, vanishes exactly where $Q$ does, and is strictly increasing where it is positive. Since $Q(0)>0$, the equation $h(p)=p_0^{\varepsilon+\eta}Q(0)$ has at most one solution, $F(p)-p$ changes sign at most once, and the projected map has exactly one fixed point.

(d) A continuous nonincreasing self-map of a compact interval has exactly one fixed point. For $p > p'$, $F(p) - F(p') \in [-L_F (p - p'), 0]$, so the damped map $D(p) = (1 - \omega)p + \omega F(p)$ satisfies $D(p) - D(p') \in [(1 - \omega - \omega L_F)(p - p'),\ (1 - \omega)(p - p')]$. Hence $D$ is a contraction with modulus $\max(|1 - \omega - \omega L_F|, |1 - \omega|) < 1$ for $\omega \in (0, 2/(1 + L_F))$.

(e) $z_i(p) = \Phi^{-1}(1 - c_i p_0/(p_i(0) p))$ is increasing in $p$; with $v$ fixed, $\alpha(p)$ is nondecreasing; by Proposition \ref{prop:computation}(b) the clearing vector $q(p)$ is nondecreasing, so $\alpha_i(p) + (Pq(p))_i$ is nondecreasing for every $i$ and $U(p) = \{i : \alpha_i(p) + (Pq(p))_i > C_i\}$ is nondecreasing.
\end{proof}

\subsection{Proofs for Section \ref{sec:shortage}}\label{app:proof_shortage}
Throughout, $T(x) = (\alpha + Px) \wedge C$ as in Section \ref{sec:model}, $q$ is its unique fixed point, $\lambda = \lambda(q) > 0$, and $\Psi(\psi) = \ind - \varsigma + P^\intercal \psi$, which maps $[0,1]^N$ into itself and is nondecreasing.

\begin{proof}[Proof of Theorem \ref{thm:shortage}]
(a) Let $(x, \psi)$ be a shortage equilibrium and fix $j$. If $x_j = C_j$ then $x_j \ge T(x)_j$. Otherwise (SE2) gives $x_j \Psi_j = \lambda_j(x)$, and $\Psi_j \le 1$ gives $x_j \ge \lambda_j(x) \ge T(x)_j$. Hence $x \ge T(x)$. $T$ is monotone, so $x \ge T(x) \ge T^2(x) \ge \cdots$, and $T^n(x) \to q$ by Theorem \ref{thm:static_solution}(a). Thus $x \ge q$, and $\lambda(x) \ge \lambda$ because $P \ge 0$. If $q_j = C_j$ then $x_j = C_j$.

(b) Let $\mathcal{L} := [q, C] \times [0,1]^N$ with the order $(x, \psi) \preceq (x', \psi')$ if and only if $x \le x'$ and $\psi \ge \psi'$; it is a complete lattice. On $\mathcal{L}$, $\lambda(x) \ge \lambda > 0$. Define $\Theta(x, \psi) := (x^+, \psi^+)$ by
\[
  x^+_j := \min\{ C_j,\ \lambda_j(x)/\Psi_j(\psi) \} \ \ (:= C_j \text{ if } \Psi_j(\psi) = 0), \qquad \psi^+_j := \min\{ 1,\ C_j \Psi_j(\psi)/\lambda_j(x) \}.
\]
$\Theta$ maps $\mathcal{L}$ into itself: $x^+_j \ge \min\{C_j, \lambda_j(x)\} \ge \min\{C_j, \lambda_j\} = q_j$. $\Theta$ is monotone for $\preceq$: $x^+$ is nondecreasing in $x$ and nonincreasing in $\psi$, and $\psi^+$ is nonincreasing in $x$ and nondecreasing in $\psi$. By Tarski's theorem \parencite{tarski1955lattice} the fixed points of $\Theta$ form a nonempty complete lattice, with a least and a greatest element.

\emph{Fixed points of $\Theta$ are the shortage equilibria.} If $(x, \psi) = \Theta(x, \psi)$ then $x_j \Psi_j = \min\{C_j \Psi_j, \lambda_j(x)\} \le \lambda_j(x)$, with equality unless $x_j = C_j$, which is (SE2); and $\psi_j \lambda_j(x) = \min\{\lambda_j(x), C_j \Psi_j\} = \min\{\lambda_j(x), x_j \Psi_j\}$, which is (SE1). Conversely, an equilibrium lies in $\mathcal{L}$ by (a). If $\Psi_j > 0$, (SE2) says $x_j \le C_j$ and $x_j \le \lambda_j(x)/\Psi_j$ with one of the two holding with equality, so $x_j = x^+_j$; if $\Psi_j = 0$ then $x_j \Psi_j = 0 < \lambda_j(x)$ forces $x_j = C_j = x^+_j$. Then $x_j \Psi_j = \min\{C_j \Psi_j, \lambda_j(x)\}$, and (SE1) gives $\psi_j \lambda_j(x) = \min\{\lambda_j(x), C_j \Psi_j\}$, that is, $\psi_j = \psi^+_j$.

\emph{Least equilibrium.} Let $\zeta_0 := (q, \ind)$. Since $\Psi(\ind) = \ind$, $\Theta(\zeta_0) = (q, \chi)$ with $\chi_j = \min\{1, C_j/\lambda_j\}$, so $\zeta_0 \preceq \Theta(\zeta_0)$ and the iterates $\Theta^n(\zeta_0)$ are $\preceq$-increasing: orders rise and fill rates fall. They converge by monotonicity and boundedness, and the limit is a fixed point because $\Theta$ is continuous on $\mathcal{L}$ ($\lambda_j(x) \ge \lambda_j > 0$, so $\min\{C_j, \lambda_j(x)/\Psi_j\}$ is continuous up to $\Psi_j = 0$). Every equilibrium $\zeta$ satisfies $\zeta_0 \preceq \zeta$, hence $\Theta^n(\zeta_0) \preceq \zeta$ for all $n$, so the limit is the least equilibrium.

(c) For $j \in U$ and any equilibrium, the fixed-point form gives $\psi_j = \min\{1, C_j \Psi_j/\lambda_j(x)\} \le C_j/\lambda_j(x) \le C_j/\lambda_j < 1$, using $\Psi_j \le 1$ and (a). If $\varsigma_j = 0$ then $\Psi_j \equiv 1$ and, by \eqref{eq:longrun}, $s^\infty_j = q_j = C_j$, so $\psi^T_j = C_j/\lambda_j$.

(d) Suppose $(q, \psi)$ is an equilibrium. By (SE2), $q_j \Psi_j(\psi) \le \lambda_j$, so (SE1) reads $\psi_j \lambda_j = q_j \Psi_j(\psi)$. Writing $s := \psi \odot \lambda$, this is $s_j = q_j (1 - \varsigma_j) + \sum_i s_i\, p_{ij} q_j/\lambda_i$, that is, $s = (\ind - \varsigma) \odot q + \tilde P^\intercal s$, which is \eqref{eq:longrun}; its solution is unique because $\rho(\tilde P) < 1$ (Theorem \ref{thm:drawdown}(a)), so $\psi = \psi^T$. For $j$ with $q_j < C_j$ we have $q_j = \lambda_j > 0$, and (SE2) requires $q_j \Psi_j = \lambda_j$, that is, $\Psi_j(\psi^T) = 1$. Conversely, $(q, \psi^T)$ satisfies (SE1) by \eqref{eq:longrun} and $q_j \Psi_j \le q_j \le \lambda_j$, and it satisfies (SE2) if and only if $\Psi_j(\psi^T) = 1$ whenever $q_j < C_j$. If $U = \emptyset$ then $L^\infty = \bm 0$ by \eqref{eq:Linf}, so $\psi^T = \ind$ and $\Psi(\psi^T) = \ind$.

(e) Example \ref{exm:bottleneck}, verified below.
\end{proof}

\begin{proof}[Proof of Proposition \ref{prop:bottleneck}]
(1) Let $(x, \psi)$ be an equilibrium. The supplier has no suppliers, so $\Psi_i = 1$, and (SE2) gives $x_i = \min\{C_i, \lambda_i(x)\}$. By Theorem \ref{thm:shortage}(a), $\lambda_i(x) \ge \lambda_i > C_i$, so $x_i = C_i$ and, by (SE1), $\psi_i = C_i/\lambda_i(x) \in (0, 1)$. Customers have no customers, so $\lambda_j(x) = \alpha_j$, and $\Psi_j = 1 - p_{ij} + p_{ij} \psi_i > 0$; hence $x_j = \min\{C_j, \alpha_j/\Psi_j(\psi_i)\}$ by (SE2). Substituting into $\psi_i \lambda_i(x) = C_i$ gives $g(\psi_i) = C_i$ with
\[
  g(\psi) := \psi \alpha_i + \sum_{j \in J} p_{ij} \min\Big\{ \psi C_j,\ \frac{\alpha_j \psi}{1 - p_{ij} + p_{ij} \psi} \Big\}, \qquad g(0) := 0.
\]
The map $\psi \mapsto \psi/(1 - p + p\psi)$ is nondecreasing on $(0, 1]$ for $p \in (0, 1]$, so each customer term is nondecreasing and continuous, and $\psi \alpha_i$ is strictly increasing because $\alpha_i > 0$. Since $g(0) = 0 < C_i$ and $g(1) = \alpha_i + \sum_j p_{ij} \min\{C_j, \alpha_j\} = \lambda_i > C_i$, there is exactly one root $\psi^* \in (0, 1)$. Conversely the profile built from $\psi^*$ satisfies (SE1)--(SE2), with customers' own fill rates $\psi_j = \min\{1, C_j \Psi_j(\psi^*)/\alpha_j\}$.

(2) Receipts are $x_j \Psi_j(\psi^*) = \min\{C_j \Psi_j(\psi^*), \alpha_j\}$.

(3) Under truthful ordering $q_j = \min\{C_j, \alpha_j\}$ and $\psi^T_i = C_i/\lambda_i$, that is, $\psi^T_i \lambda_i = C_i$. For every $\psi$, $\min\{C_j, \alpha_j/\Psi_j(\psi)\} \ge q_j$, so $g(\psi) \ge \psi \lambda_i$ and $g(\psi^T_i) \ge C_i = g(\psi^*)$; hence $\psi^* \le \psi^T_i$. If some customer has $C_j > \alpha_j$ then, since $\Psi_j(\psi^T_i) < 1$, $\min\{C_j, \alpha_j/\Psi_j(\psi^T_i)\} > \alpha_j = q_j$ and $g(\psi^T_i) > C_i$, so $\psi^* < \psi^T_i$; if no customer has slack, $g(\psi) = \psi \lambda_i$ and $\psi^* = \psi^T_i$. The supplier's own market receives $\psi \alpha_i$ and a customer with $C_j \le \alpha_j$ receives $C_j \Psi_j(\psi)$, both nondecreasing in $\psi$.

(4) In the game, (SE1) gives $\psi_i = \min\{1, x_i/\lambda_i(x)\}$ for the supplier, with $\lambda_i(x) \ge \alpha_i > 0$. The supplier receives $x_i$, and its payoff is strictly concave in $x_i$ with peak at $\lambda_i(x)$, which does not depend on $x_i$; its unique best response is $\min\{C_i, \lambda_i(x)\} > 0$. Customer $j$ receives
\[
  R_j(x_j) = x_j (1 - p_{ij}) + p_{ij}\, x_j \min\Big\{ 1,\ \frac{x_i}{a_j + p_{ij} x_j} \Big\}, \qquad a_j := \alpha_i + \sum_{k \ne j} p_{ik} x_k \ge \alpha_i > 0,
\]
which, for $x_i > 0$, is continuous and strictly increasing in $x_j$ with $R_j(0) = 0$. Its payoff is strictly concave in receipts with peak at $\alpha_j$, so its unique best response is the $x_j$ with $R_j(x_j) = \alpha_j$ if $R_j(C_j) \ge \alpha_j$, and $C_j$ otherwise. A profile is therefore a Nash equilibrium if and only if every node satisfies (SE2) at the fill rates generated by (SE1), that is, if and only if it is a shortage equilibrium, which is unique by (1).
\end{proof}

\begin{proof}[Verification of Example \ref{exm:bottleneck}]
\emph{Truthful ordering.} $T(q) = q$ for $q = (5, 20, 20, 10)$, since $\lambda(q) = (\tfrac12 \cdot 20,\ \tfrac12 \cdot 20 + 10,\ 20,\ 10) = (10, 20, 20, 10)$ and $C = (5, 22.5, 50, 10)$. Suppliers 1 and 2 have no suppliers, so $\psi^T_1 = 5/10 = 1/2$ and $\psi^T_2 = 20/20 = 1$; $\Psi_3 = \tfrac12 \cdot \tfrac12 + \tfrac12 \cdot 1 = 3/4$ and $\Psi_4 = 1$, giving receipts 15 and 10.

\emph{Equilibrium.} In any equilibrium $x_4 = \min\{10, 10/\psi_2\} = 10$ and $x_3 \ge q_3 = 20$, so $\lambda_1(x) = x_3/2 \ge 10 > 5$, $x_1 = 5$ and $\psi_1 = 10/x_3$. Also $\lambda_2(x) = x_3/2 + 10$, $x_2 = \min\{22.5, \lambda_2(x)\}$ and $\psi_2 = \min\{1, 45/(x_3 + 20)\}$, which equals one if and only if $x_3 \le 25$. Finally $x_3 = \min\{50, 40/(\psi_1 + \psi_2)\}$. If $x_3 \le 25$ then $x_3 (10/x_3 + 1) = 40$ gives $x_3 = 30$, a contradiction. If $25 < x_3 < 50$, then $x_3 (10/x_3 + 45/(x_3 + 20)) = 40$ gives $45 x_3 = 30 (x_3 + 20)$, that is, $x_3 = 40$. If $x_3 = 50$, then $40/(\psi_1 + \psi_2) = 40/(1/5 + 9/14) < 50$, a contradiction. Hence the equilibrium is unique: $x = (5, 22.5, 40, 10)$, $\psi_1 = 1/4$, $\psi_2 = 22.5/30 = 3/4$, $\Psi_3 = 1/2$, $\Psi_4 = 3/4$; customer 3 receives $40 \cdot \tfrac12 = 20$ and customer 4 receives $10 \cdot \tfrac34 = 7.5$. Node $2 \notin U$ orders at capacity and rations; node 4 has no supplier in $U$ and no supplier downstream of $U$.
\end{proof}

\begin{proof}[Proof of Theorem \ref{thm:network}]
(a) By (SE2), $x_j \Psi_j(\psi) \le \lambda_j(x)$, so (SE1) reads $\psi_j \lambda_j(x) = x_j \Psi_j(\psi)$ for every $j$, that is,
\[
  \psi_j \alpha_j + \sum_k \psi_j\, p_{jk} x_k = (1 - \varsigma_j) x_j + \sum_i \psi_i\, p_{ij} x_j.
\]
Summing over $j$, the two double sums coincide and \eqref{eq:conservation} follows.

(b) Summing \eqref{eq:longrun} over nodes and using $\sum_j \tilde P_{ij} = 1 - \alpha_i/\lambda_i$ gives $\sum_i \alpha_i s^\infty_i/\lambda_i = \sum_j (1 - \varsigma_j) q_j$, that is, $\sum_j \alpha_j \psi^T_j = \sum_j (1 - \varsigma_j) q_j$. By Theorem \ref{thm:shortage}(a), $x \ge q$, so by \eqref{eq:conservation} $\sum_j \alpha_j \psi_j - \sum_j \alpha_j \psi^T_j = \sum_j (1 - \varsigma_j)(x_j - q_j) \ge 0$, with strict inequality if and only if $x_j > q_j$ for some $j$ with $\varsigma_j < 1$.

(c) Let $(\underline{x}, \overline{\psi})$ be the least and $(\overline{x}, \underline{\psi})$ the greatest equilibrium (Theorem \ref{thm:shortage}(b)), so $\underline{x} \le \overline{x}$ and $\underline{\psi} \le \overline{\psi}$. By \eqref{eq:conservation},
\[
  \sum_j \alpha_j \overline{\psi}_j = \sum_j (1 - \varsigma_j)\, \underline{x}_j \le \sum_j (1 - \varsigma_j)\, \overline{x}_j = \sum_j \alpha_j \underline{\psi}_j \le \sum_j \alpha_j \overline{\psi}_j.
\]
Both inequalities are therefore equalities, and since every term is ordered, $\overline{\psi}_j = \underline{\psi}_j$ whenever $\alpha_j > 0$ and $\underline{x}_j = \overline{x}_j$ whenever $\varsigma_j < 1$. Every equilibrium $(x, \psi)$ satisfies $\underline{x} \le x \le \overline{x}$ and $\underline{\psi} \le \psi \le \overline{\psi}$, which gives the claim.

(d) Let $\alpha > 0$. By (c) all equilibria have the same $\psi$, so $x := \underline{x}$ and $x' := \overline{x}$ both satisfy $x_j = \min\{C_j, \lambda_j(x)/\Psi_j(\psi)\}$ with this $\psi$ (and $x_j = C_j$ where $\Psi_j = 0$), and $x \le x'$. If $C_j = 0$ then $x_j = x'_j = 0$. Let $J^+ := \{j : C_j > 0\}$. For $j \in J^+$, $x_j > 0$: either $\Psi_j = 0$ and $x_j = C_j$, or $x_j \ge \min\{C_j, \alpha_j/\Psi_j\} > 0$. Let $t := \min_{j \in J^+} x_j/x'_j \in (0, 1]$ and suppose $t < 1$. Pick $j \in J^+$ with $x_j = t x'_j$. Then $x_j < x'_j \le C_j$, so $\Psi_j > 0$ and $x_j = \lambda_j(x)/\Psi_j$. Since $x_k \ge t x'_k$ for every $k$ and $\alpha_j > 0$,
\[
  \lambda_j(x) = \alpha_j + \sum_k p_{jk} x_k \ge \alpha_j + t \sum_k p_{jk} x'_k > t\, \lambda_j(x'),
\]
so $x_j > t \lambda_j(x')/\Psi_j \ge t x'_j$, a contradiction. Hence $t = 1$, $x \ge x'$, and $x = x'$. The least and greatest equilibria coincide, so the equilibrium is unique. For a supply chain network, $\alpha_i = v_i + z_i \sigma_i > 0$ by Assumption \ref{defn:hedging}.

(e) Let $\zeta_0 := (q, \ind)$. Then $\Theta(\zeta_0) = (q, \chi)$, because $\Psi(\ind) = \ind$ and $\min\{C_j, \lambda_j\} = q_j$. Next $\Theta^2(\zeta_0) = (x^2, \psi^2)$ with $x^2_j = \min\{C_j, \lambda_j/\Psi_j(\chi)\}$, so that $\lambda_i(x^2) = \hat\lambda_i$, and the fill-rate component of $\Theta^3(\zeta_0)$ is $\psi^3_i = \min\{1, C_i \Psi_i(\psi^2)/\lambda_i(x^2)\} \le C_i/\hat\lambda_i$. By the proof of Theorem \ref{thm:shortage}(b), every equilibrium $(x, \psi)$ satisfies $\Theta^3(\zeta_0) \preceq (x, \psi)$, so $\psi_i \le \psi^3_i < 1$ when $\hat\lambda_i > C_i$. Then $x_i \Psi_i = \psi_i \lambda_i(x) < \lambda_i(x)$, and (SE2) gives $x_i = C_i$. The formula for $\Psi_j(\chi)$ is \eqref{eq:Psi} with $\chi_k = 1$ for $k \notin U$. If $i \notin U$ then $\lambda_i \le C_i < \hat\lambda_i$, so some customer $j$ of $i$ has $\min\{C_j, \lambda_j/\Psi_j(\chi)\} > q_j = \min\{C_j, \lambda_j\}$, which requires $C_j > \lambda_j$ and $\Psi_j(\chi) < 1$, that is, $p_{kj} > 0$ for some $k \in U$. For the converse statement, in any equilibrium $\psi_i = \min\{1, C_i \Psi_i(\psi)/\lambda_i(x)\}$, and $C_i \Psi_i(\psi) \ge (1 - \varsigma_i) C_i \ge \alpha_i + \sum_j p_{ij} C_j \ge \lambda_i(x)$, so $\psi_i = 1$.

(f) The statement for $\zeta^n$ is Theorem \ref{thm:shortage}(b). The pair $(C, \bm 0)$ is the greatest element of the lattice $\mathcal{L}$ of that proof, so $\Theta(C, \bm 0) \preceq (C, \bm 0)$ and, $\Theta$ being monotone, $\zeta_n$ is $\preceq$-decreasing; it converges, and its limit is a fixed point because $\Theta$ is continuous on $\mathcal{L}$. Every equilibrium $\zeta$ satisfies $\zeta \preceq (C, \bm 0)$, hence $\zeta = \Theta^n(\zeta) \preceq \zeta_n$ for all $n$, so the limit is the greatest equilibrium; likewise $\zeta^n \preceq \zeta$. A sweep requires $\lambda(x) = \alpha + Px$ and $\Psi(\psi) = \ind - \varsigma + P^\intercal \psi$, two products with $P$, and $2N$ scalar operations. If $\alpha > 0$, the least and greatest equilibria coincide by (d).
\end{proof}

\begin{proof}[Proof of Proposition~\ref{prop:nash}]
Throughout, $\lambda(x)\ge\alpha>0$, and $\lambda_j(x)=\alpha_j+\sum_k p_{jk}x_k$ does not depend on $x_j$ because $p_{jj}=0$.

(a) Write $s_k:=\psi_k\lambda_k(x)$ for the deliveries of node $k$. Then (SE1) reads $s=S(s)$ with
\[
S(s)_k:=\min\Bigl\{\lambda_k(x),\;x_k(1-\varsigma_k)+\sum_i s_i\,\frac{p_{ik}x_k}{\lambda_i(x)}\Bigr\}.
\]
$S$ maps $[0,\lambda(x)]$ into itself, and since $\sum_k p_{ik}x_k/\lambda_i(x)=1-\alpha_i/\lambda_i(x)$,
\[
\|S(s)-S(s')\|_1\le\sum_i|s_i-s_i'|\Bigl(1-\frac{\alpha_i}{\lambda_i(x)}\Bigr)\le\Bigl(1-\min_i\frac{\alpha_i}{\lambda_i(C)}\Bigr)\|s-s'\|_1 .
\]
The modulus is below one and does not depend on $x$, and $S$ is continuous in $x$. Hence the fixed point is unique and continuous in $x$.

\emph{Support.} Let $D(x)$ be the least set of nodes such that $k\in D(x)$ whenever $x_k>0$ and either $\varsigma_k<1$ or $p_{ik}>0$ for some $i\in D(x)$. Then $\{k:s_k>0\}=D(x)$. Indeed, every node of $D(x)$ has $s_k>0$ by induction along the construction of $D(x)$. Conversely, the vector that solves $s=S(s)$ on $D(x)$ and is zero elsewhere is a fixed point of $S$, because a node outside $D(x)$ has $x_k=0$, or $\varsigma_k=1$ and no supplier in $D(x)$; by uniqueness it is the fixed point.

(b) Fix $j$ and the orders of the other nodes, write $t:=x_j$, and regard $\psi$, $s$ and $R_j=t\,\Psi_j(\psi)$ as functions of $t\in[0,C_j]$. Continuity follows from (a), and $R_j(0)=0$. For $i\ne j$, $\lambda_i=a_i+p_{ij}t$ with $a_i:=\alpha_i+\sum_{k\ne j}p_{ik}x_k>0$ independent of $t$. Let $\pi_i(t):=p_{ij}t/\lambda_i\in[0,1)$, the share of the requirement of $i$ that is owed to $j$; it is nondecreasing in $t$, and strictly increasing if $p_{ij}>0$.

\emph{Regimes.} Let $B(t):=\{k:x_k\Psi_k(\psi(t))<\lambda_k\}$ be the set of nodes that ration. For a set $B$ of nodes and $i\in B$, let $W_{ik}:=p_{ik}x_k/\lambda_i$ be the share of the deliveries of $i$ that goes to node $k$; the rows of $W$ sum to $1-\alpha_i/\lambda_i<1$. Let $\psi^B(t)$ be given by $\psi^B_k=1$ for $k\notin B$ and, on $B$, by the linear system $\psi_k\lambda_k=x_k\Psi_k(\psi)$. In terms of $s_B$ this system reads $(I_B-W_{BB}^{\intercal})s_B=b_B$ with $b_k:=x_k(1-\varsigma_k+\sum_{i\notin B}p_{ik})\ge 0$; it is nonsingular with a nonnegative inverse, and its coefficients are rational functions of $t$, so $\psi^B$ is a nonnegative rational function of $t$. By the uniqueness in (a), $B(t)=B$ if and only if $x_k\Psi_k(\psi^B(t))<\lambda_k$ for $k\in B$ and $x_k\Psi_k(\psi^B(t))\ge\lambda_k$ for $k\notin B$. These are finitely many inequalities between rational functions of $t$, so each set $\{t:B(t)=B\}$ is a finite union of points and intervals. Hence $[0,C_j]$ is a finite union of points and of intervals on each of which $B(t)$ is constant, and because $R_j$ is continuous it suffices to prove the monotonicity claims on one such interval $I$.

\emph{Formula.} Let $B$ be the rationing set on $I$ and $B':=B\setminus\{j\}$. For $k\in B'$ put $\beta_k:=x_k(1-\varsigma_k+\sum_{i\notin B}p_{ik})$ and $\mu_k:=p_{jk}x_k/\lambda_j$, and let $c_j:=1-\varsigma_j+\sum_{i\notin B}p_{ij}$; none of these depends on $t$. On $B$, (SE1) reads $s_k=x_k\Psi_k(\psi)$, that is,
\[
s_k=\beta_k+\sum_{i\in B'}s_iW_{ik}+\mathbf{1}\{j\in B\}\,s_j\mu_k\quad(k\in B'),\qquad R_j=t\,c_j+\sum_{i\in B'}s_iW_{ij},
\]
with $s_j=R_j$ if $j\in B$. Let $h\in[0,1]^{B'}$ be the solution of $h_k=W_{kj}+\sum_{l\in B'}W_{kl}h_l$: the share of a unit delivered by $k$ that reaches $j$ when it is passed on pro rata inside $B'$. Solving the first relation for $s_{B'}$ and substituting into the second gives
\[
R_j=t\,c_j+\beta\cdot h\quad\text{if }j\notin B,\qquad R_j=\frac{t\,c_j+\beta\cdot h}{1-\mu\cdot h}\quad\text{if }j\in B,
\]
where $\mu\cdot h\le\sum_k\mu_k\le 1-\alpha_j/\lambda_j<1$.

\emph{$h$ is nondecreasing in $t$.} For $i\in B'$ we have $W_{ij}=\pi_i(t)$ and $W_{ik}=(1-\pi_i(t))\,m_{ik}$ for $k\ne j$, with $m_{ik}:=p_{ik}x_k/a_i$ independent of $t$ and $\sum_{k\ne j}m_{ik}=1-\alpha_i/a_i<1$. Hence $h=\Phi_t(h):=\pi(t)+(1-\pi(t))\odot(mh)$, with $m$ restricted to $B'\times B'$. The map $\Phi_t$ is monotone and a contraction in the maximum norm. For $t<t'$ in $I$ and $h\in[0,1]^{B'}$,
\[
\Phi_{t'}(h)-\Phi_t(h)=(\pi(t')-\pi(t))\odot(\mathbf{1}-mh)\ge 0 .
\]
So $h(t)=\Phi_t(h(t))\le\Phi_{t'}(h(t))$, and iterating $\Phi_{t'}$ from $h(t)$ gives a nondecreasing sequence that converges to $h(t')$. Thus $h(t)\le h(t')$. In both formulas the numerator is nonnegative and nondecreasing in $t$ and the denominator is positive and nonincreasing, so $R_j$ is nondecreasing on $I$.

\emph{Dichotomy.} The set $D(x)$ of the support step is the same for every $t>0$, and $R_j(t)>0$ if and only if $s_j(t)>0$, that is, $j\in D(x)$. So $R_j$ is zero for every $t$ or positive for every $t>0$. Suppose it is positive and let $0<t<t'$ in $I$. If $c_j>0$ the numerator is strictly increasing, and $R_j(t)<R_j(t')$. If $c_j=0$, then $\beta\cdot h(t)>0$, so $\beta_k>0$ and $h_k(t)>0$ for some $k\in B'$, and there is a chain $k=k_0,\dots,k_r$ in $B'$ with $m_{k_lk_{l+1}}>0$ for every $l$ and $p_{k_rj}>0$. Since $\pi_{k_r}$ is strictly increasing and $(mh)_{k_r}<1$, $h_{k_r}(t')\ge\Phi_{t'}(h(t))_{k_r}>h_{k_r}(t)$. If $h_{k_{l+1}}(t')>h_{k_{l+1}}(t)$, then
\[
h_{k_l}(t')=\Phi_{t'}(h(t'))_{k_l}\ge\Phi_t(h(t'))_{k_l}>\Phi_t(h(t))_{k_l}=h_{k_l}(t),
\]
because $1-\pi_{k_l}(t)>0$ and $m_{k_lk_{l+1}}>0$. By induction $h_k(t')>h_k(t)$, and therefore $R_j(t')>R_j(t)$.

(c) The payoff of node $j$ is a strictly concave function of its receipts with maximum at $\lambda_j(x)$ (Section~\ref{sec:model}), and $\lambda_j(x)$ does not depend on $x_j$.

\emph{A shortage equilibrium is a Nash equilibrium.} Let $(x,\psi)$ be a shortage equilibrium; by (a), $\psi=\psi(x)$. Fix $j$. By (SE2), either $R_j(x)=\lambda_j(x)$, and $x_j$ attains the maximum of the payoff, or $x_j=C_j$ and $R_j(x)<\lambda_j(x)$. In the second case (b) shows that every order in $[0,C_j]$ yields receipts of at most $R_j(x)<\lambda_j(x)$, a range on which the payoff is increasing. So no node gains by deviating.

\emph{A Nash equilibrium is a shortage equilibrium.} Let $x$ be a Nash equilibrium. We first show that $x>0$ and that no receipt function vanishes identically. By Assumption~\ref{defn:open}, every node $j$ is the end of a chain $j_0,\dots,j_d=j$ with $\varsigma_{j_0}<1$ and $p_{j_lj_{l+1}}>0$ for every $l$; let $d(j)$ be the least such $d$. We argue by induction on $d(j)$. If $d(j)=0$, then $R_j\ge(1-\varsigma_j)x_j$, so $R_j$ is strictly increasing in $x_j$ by (b); since the payoff is strictly increasing in receipts below $\lambda_j(x)>0$, the order $x_j=0$ is not a best response, and $x_j>0$. If $d(j)=d\ge 1$, take a chain of length $d$. Its nodes $j_0,\dots,j_{d-1}$ differ from $j$ and have $d(\cdot)<d$, so their orders are positive by the induction hypothesis. By the support step of (a), $\psi_{j_{d-1}}>0$ whatever node $j$ orders. Hence $R_j\ge x_j\,p_{j_{d-1}j}\,\psi_{j_{d-1}}>0$ for $x_j>0$, $R_j$ is strictly increasing in $x_j$ by (b), and $x_j>0$ as before.

Since every $R_j$ is strictly increasing in $x_j$, the best response of node $j$ is unique: it is the order at which $R_j=\lambda_j(x)$ if the receipts at $x_j=C_j$ are at least $\lambda_j(x)$, and it is $C_j$ otherwise. In both cases (SE2) holds, and (SE1) holds by the definition of $\psi(x)$. So $(x,\psi(x))$ is a shortage equilibrium. Existence and uniqueness follow from Theorem~\ref{thm:network}(d).
\end{proof}
\section{The Stocks-and-Flows Reading of the Safety Term}\label{app:stockflow}
\begin{rem}[Stocks and flows]\label{rem:stock_flow}
The requirement $\lambda = \alpha + Pq$ of \eqref{eq:flow} is a rate, and $\alpha_i = v_i + z_i\sigma_i$ includes the safety term. In the dynamic model the safety term is therefore a flow: a node with $z_i > 0$ is required to deliver $z_i\sigma_i$ per unit of time to its own market above mean demand. We read this as precautionary demand, that is, purchases above mean consumption that the node's market absorbs and that do not accumulate in $Z_i$. This is why hedged nodes have zero drift in \eqref{eq:free} although their order targets exceed mean demand. The case $z = 0$ is the physical baseline, in which requirements are mean flows. With $z > 0$ the requirements also carry precautionary orders, and stocks are drawn faster. Alternatively, the safety term can be treated as a level: the dynamics are run at $\alpha = v$, and the safety stock, $z_i\sigma_i$ over one review period, is added to $Z_i(0)$. Theorem~\ref{thm:drawdown} and Propositions~\ref{prop:exposure}--\ref{prop:tractable} use only Assumption~\ref{defn:open}, $\alpha \ge 0$, $C \ge 0$ and $\lambda > 0$, so they hold in both readings whenever $\lambda > 0$ at $\alpha = v$. In the second reading the unhedged set of the dynamics is defined by mean flows, and by Proposition~\ref{prop:exposure}(iii) a larger $z$ weakly delays every exhaustion time. The two readings coincide at $z = 0$.
\end{rem}

\section{Limitations of the Numerical Experiment}\label{app:limits}
The experiment measures exposure through the crude trade network, and the results have the following limitations. First, the data cover internationally traded crude. Domestic production consumed at home and trade in refined products are not represented, so the loss shares in Table~\ref{tab:downstream_flow_comparison} are relative to crude imports; they are not shares of consumption. Second, the routing assumptions affect the exposure in both directions. All Gulf exports are routed through the Hormuz node and the Saudi and Emirati pipelines that bypass the strait are not modelled, so the loss shares and the fall in Gulf shipments are upper bounds in this respect. Conversely, no importer in the data reports crude of Iranian origin, so Iranian exports, which leave the Gulf through the strait, are absent from the network and the exposure of the buyers of that crude is understated. China's flows enter at their 2024 levels, the last year in which it reports. The network also does not adapt: importers do not switch suppliers, no flow is rerouted, and the closure lasts the full horizon. Third, the exhaustion times in Table~\ref{tab:buffer_debilitation_comparison} depend on the initial stocks, while the weeks per 30 days of cover in Table~\ref{tab:downstream_flow_comparison} depend only on the loss shares. For the 16 IEA countries the stocks are the reported days of net-import cover, which we apply to crude net imports. For China, India and the United States we use the 90-day benchmark, which is a scenario and was not measured, and the timing of the price increases in the endogenous-price variant depends on it as well. Finally, the order caps in Section~\ref{sec:exp_game} are also scenarios. The aggregate gain from order inflation is set by the 5\% headroom assumed for the exporters, so it is not an estimate of spare production capacity, and since the sourcing shares are fixed, buyers cannot respond to rationing by switching suppliers.

\section{Data Rules and Constants}\label{app:parameters}
\begin{table}[ht]
\centering
\caption{Data rules and constants of the experiment.}\label{tab:parameters}
\begin{tabular}{ll}
\hline
Item & Value \\
\hline
Trade data & UN Comtrade, HS 2709, monthly, Nov.\ 2021--Dec.\ 2025 \\
Unit-value filter & \$35--\$140 per barrel \\
Active nodes & 163 \\
Crude in the network & 46.6 mb/d \\
Total net demand & 39.3 mb/d \\
Baseline flow through Hormuz & 15.8 mb/d \\
Capacity & $\max\{\text{historical},\ 1.05\times\text{baseline requirement}\}$ \\
Stock, IEA countries (16) & reported days of net-import cover \\
Stock, other net importers & benchmark, 90 days of net demand \\
Stock, nodes without net demand & 4 weeks of throughput \\
Stock, transit nodes & 0 \\
Safety level $z$ & 0 (baseline), 1.645, $\bar z=2.576$ \\
Shock & Hormuz capacity cut to 10\% of its baseline requirement \\
Horizon & 52 weeks \\
Price variant & $\varepsilon=\eta=0.10$, $p_0=\$75$ per barrel, $\kappa=5$ \\
\hline
\end{tabular}
\end{table}

\section{Future Research Directions}\label{app:future}
We note three future research directions:

First, the routing matrix $P(t)$ changes only through the rebalancing rule \eqref{eq:rebalancing}. Modeling how nodes choose new suppliers, for example through a matching market, is left for future work.

Second, the framework uses a two-moment distribution policy (mean and variance) for the safety-stock right-hand side parameters ($\alpha$). This is accurate for mature, high-volume commodities such as oil or industrial chemicals, but not for highly intermittent, lumpy, or fat-tailed demand patterns typical of spare parts or high-tech components. Extending the clearing transformations to generalized non-linear inventory operators would broaden the industrial scope.

Finally, the model endogenizes price feedback but treats supplier operational costs ($c_i$) as constant parameters. Under sustained global shocks, shipping freight rates, maritime insurance premiums, and labor inputs increase sharply. Adding cost-push effects to the price model is left for future work.

\section{A Worked Example}\label{app:example}
\begin{exm}[{\bf Rare earth supply chain}]
    A supply chain network composed of a producer (node 1) and a seller/country (node 2). The external average demand $v_2=100$ tons per year, the external standard deviation of the demand $w_2=10$ tons. The seller buys all of its product from the producer, that is $p_{12}=1$, the routing matrix is:
    \begin{equation*}
        P=\begin{pmatrix}
            0 & 1 \\
            0 & 0
        \end{pmatrix},
    \end{equation*}
    With a safety level $z_1=z_2=1.645$ what are the optimal production quantities in this network? While maintaining the same average demand, how much an increase of 1 ton of the standard deviation of the demand increases the quantities for the producer?

    \noindent For the first question, we have the following quantities:
    \begin{align*}
        q_2&=v_2+z_2w_2=100+1.645(10)=116.45\\
        q_1&=p_{12}q_2+z_1\sqrt{p_{12}^2\sigma^2_2}=116.45+1.645(10)=132.9
    \end{align*}
    Even when the network only has 2 nodes, the producer has to produce way above node 2 to be optimally hedged against the demand fluctuation of the seller.

    \noindent For the second question, $(I - P)^{-1} = \begin{pmatrix} 1 & 1 \\ 0 & 1 \end{pmatrix}$. Since $P^2 = 0$, $N_\sigma = I + P \odot P = (I - P)^{-1}$, and from \eqref{eq:sens_sigma}, $\partial\sigma_1/\partial w_2 = w_2 [N_\sigma]_{12}/\sigma_1 = 10 \cdot 1/10 = 1$ and $\partial\sigma_2/\partial w_2 = 1$. From \eqref{eq:sens_w},
    \begin{equation}\label{eq:example_sens}
        \frac{\partial q_1}{\partial w_2} = [(I - P)^{-1}]_{12} z_2 \frac{\partial\sigma_2}{\partial w_2} + [(I - P)^{-1}]_{11} z_1 \frac{\partial\sigma_1}{\partial w_2} = 1 \cdot 1.645 \cdot 1 + 1 \cdot 1.645 \cdot 1 = 3.29.
    \end{equation}
    \noindent An increase of one ton in the seller's demand volatility raises the producer's quantity by 3.29 tons: 1.645 from the seller's larger order and 1.645 from the producer's own safety stock against that order.
\end{exm}

\begin{figure}[ht]
\centering
\includegraphics[width=\textwidth]{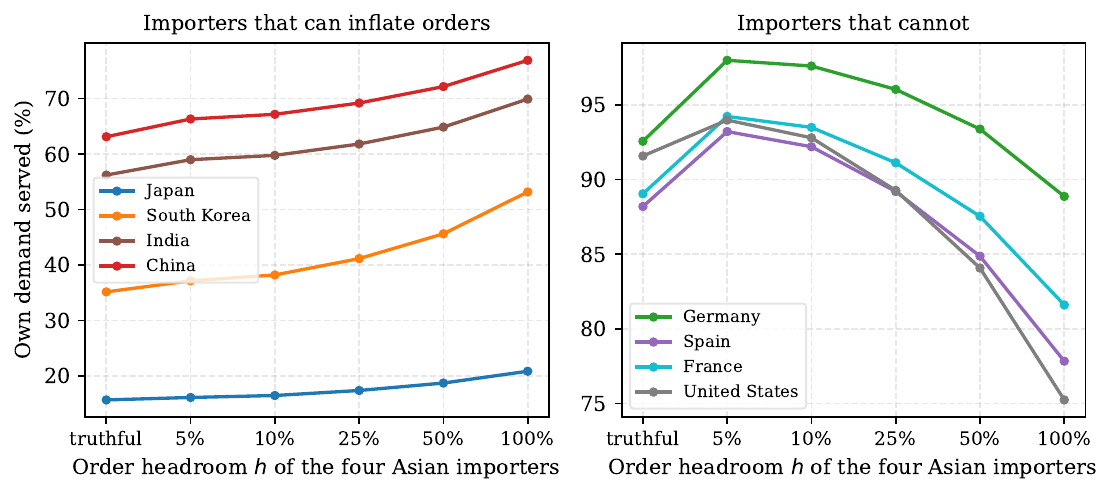}
\caption{Share of own demand served in the long run under truthful ordering and in the least shortage equilibrium, as a function of the order headroom $h$ of Japan, South Korea, India and China. All other nodes have a headroom of 5\%.}\label{fig:shortage_game}
\end{figure}

\section{Computation of the Experiment}\label{app:computation}

\paragraph{Stock paths.} The network has $N=163$ nodes. The exact 52-week path at $\bar z$ is obtained in 33 passes of the sequential algorithm, against a worst-case bound of $N+1=164$.

\paragraph{Equilibrium selection.} The Hormuz node and the pure exporters have no market of their own, so Theorem~\ref{thm:network}(d) and Proposition~\ref{prop:nash} do not apply. By Theorem~\ref{thm:network}(c), the fill rates of all nodes with a market of their own and aggregate served demand are the same in every equilibrium. We compute the least and the greatest equilibrium with the two monotone iterations of Theorem~\ref{thm:network}(f) and report the least one. Table~\ref{tab:shortage_game} reports served demand, the number of strategic bottlenecks and the number of sweeps of each iteration, and Figure~\ref{fig:shortage_game} shows the share of own demand served at the eight importers of Table~\ref{tab:downstream_flow_comparison}.

\paragraph{Checks.} The least and the greatest equilibrium agree on the fill rates of all nodes with a market of their own, as Theorem~\ref{thm:network}(c) requires, and served demand equals own production to numerical precision in every run (Theorem~\ref{thm:network}(a)). The two monotone iterations need between 24 and 835 sweeps to reach a tolerance of $10^{-13}$; they are slowest at $h=5\%$ and $h=10\%$, close to the boundary case. At $h=50\%$ we varied the order of each of the eight importers on a grid of 41 values and recomputed all fill rates: the receipts of each importer are nondecreasing in its own order, and no importer can bring its receipts closer to its requirement by deviating. This is the property established in Proposition~\ref{prop:nash}, although the proposition does not cover this network.

\section{The Self-Multiplier}\label{app:self_multiplier}
\begin{rem}[Self-multiplier]\label{rem:self_multiplier}
The $i$-th diagonal entry of the network multiplier satisfies $[(I_\calh - P_\calh)^{-1}]_{ii} = 1 + \sum_{k\ge1} [P_\calh^k]_{ii} \ge 1$, with strict inequality if and only if $i$ lies on a directed cycle within $\calh$. The excess over one is the compounding of $i$'s own order target through orders that return to $i$. This is a statement about the sensitivity of a deterministic target to its own parameters; it is not an order-to-demand variance ratio, and under the pass-through process \eqref{eq:passthrough} that ratio is identically one at every node, consistent with the $\ell=0$ case of \textcite{chen2000quantifying}.
\end{rem}

\section{Computation on Large Networks}\label{app:computation_networks}
Write $m$ for the number of nonzero entries of $P$, that is, the number of trade links.

\begin{pro}[Tractability]\label{prop:tractable}
Under the conditions of Theorem \ref{thm:drawdown}:
\begin{enumerate}
\item[(i)] \emph{Static equilibrium.} The clearing vector is the solution of the linear program \eqref{eq:lp}, which has $N$ variables, $N$ coupling constraints with at most $m + N$ nonzero coefficients, and simple bounds. Equivalently, it is the limit of the iteration $q \leftarrow T(q)$, each sweep of which costs $O(N + m)$ operations and contracts the error by the fixed factor $\gamma < 1$ of Theorem \ref{thm:static_solution}(a).
\item[(ii)] \emph{Drawdown path.} Algorithm \ref{alg:sequential_dynamics} returns the exact solution of the Skorokhod problem \eqref{eq:skorokhod} on $[0, T]$. The while loop is executed at most $N + 1$ times, because every pass that does not terminate at line \ref{alg1:break} empties at least one stocked node. Each pass evaluates the closed-form expressions of lines \ref{alg1:hedged}, \ref{alg1:unhedged} and \ref{alg1:rates}. There is no time discretization, and the number of passes does not depend on the horizon $T$. Lines \ref{alg1:partition}--\ref{alg1:unhedged} do not depend on $Z$ and need only be evaluated once.
\item[(iii)] \emph{Long-run exposure.} The long-run shortfall rates $L^\infty$ and fill rates $s^\infty/\lambda$ are given by the closed form \eqref{eq:Linf}, without running Algorithm \ref{alg:sequential_dynamics}. The set of exposed nodes of Proposition \ref{prop:exposure}(i) is found by a search from $U$ along the trade links in $O(N + m)$ operations.
\item[(iv)] \emph{Pricing and rebalancing.} The while loop of Algorithm \ref{alg:enhanced_strategic_dynamics} is also executed at most $N + 1$ times, and each evaluation of the price map requires one evaluation of its lines \ref{alg2:targets}--\ref{alg2:price}.
\end{enumerate}
\end{pro}

\section{Scope and Open Questions of the Shortage Game}\label{app:scope}
\begin{rem}[Scope and open questions]\label{rem:shortage_scope}
(i) The equilibrium notion takes fill rates as given. By Proposition~\ref{prop:nash} this is without loss when $\alpha>0$, $C>0$ and no node supplies itself: the shortage equilibrium is then the unique Nash equilibrium of the game in which every node accounts for the effect of its own order on all fill rates, in any network, with or without cycles. The proposition does not cover nodes without a market of their own. In Example~\ref{exm:bottleneck}, where $\alpha_1=\alpha_2=0$, a direct computation shows that the equilibrium is also a Nash equilibrium. (ii) When some $\alpha_i = 0$, as for pure transit nodes or when order targets are floored at zero in Section \ref{sec:rebalancing}, the fill rate of such a node need not be pinned down: if a supplier with $\alpha_i = 0$ and capacity exactly equal to total need serves two single-sourced customers, every fill rate between one and the level at which the customers' order caps bind is an equilibrium. Even then, by Theorem \ref{thm:network}(c), the fill rates of all nodes with a market of their own and aggregate deliveries are the same in every equilibrium. We establish monotone convergence of the iterations in Theorem \ref{thm:network}(f) but not a rate. (iii) The game is static. Under inflated orders $x$ may exceed $\lambda(x)$, so the reflection matrix of Section \ref{sec:dynamic} changes; the transient with strategic orders is left for future work, as is the question of which rationing rules restore truthful ordering in a network.
\end{rem}

\section{Data of the Numerical Experiment}\label{app:data}

\paragraph{Trade flows.} We use monthly bilateral import and export records for crude petroleum (HS 2709) from UN Comtrade, November 2021 to December 2025. Records whose implied unit value lies outside \$35--\$140 per barrel are discarded as reporting errors. The flow on each edge is taken from the importer's own reports: its 2025 average if it reported at least six months of that year, and otherwise the average of its last twelve reporting months. If the importer does not report, we use the exporter's (mirror) reports, excluding exports declared to unspecified partner areas. China's reports end in December 2024, so its imports enter at their 2024 level of 11.1 mb/d. The resulting network has 163 active nodes and carries 46.6 mb/d of crude.

\paragraph{Chokepoints.} Transit nodes represent the Strait of Hormuz, the Strait of Malacca, and the Suez Canal with the SUMED pipeline. Every export of Saudi Arabia, the United Arab Emirates, Iraq, Kuwait, Qatar, Bahrain and Iran is routed through the Hormuz node. From there, flows to importers east of the Strait of Malacca continue through the Malacca node, flows to European importers continue through the Suez node, and all other flows are direct. Transit nodes have no demand of their own and hold no stock. In the data the Hormuz node carries 15.8 mb/d, which is 34\% of the crude in the network.

\paragraph{Model inputs.} The routing proportions $P$ are the edge flows divided by each node's baseline throughput, so that clearing the model without safety orders reproduces the observed flows exactly. The net demand $v_i$ of a node is the excess of its imports over its exports when this is positive, and zero otherwise; total net demand is 39.3 mb/d. Demand volatility is a robust coefficient of variation of the node's monthly import series times its net demand (0.08 to 0.13 for China, India, Japan and South Korea, and 0.27 for nodes without a usable series), and the standard deviation of each node's requirement follows from Theorem~\ref{thm:variance}. The capacity of a node is the larger of its historical capacity and 1.05 times its baseline requirement, so that no node is unhedged before the shock. Historical capacity is below the baseline requirement at 84 of the 163 nodes. This does not affect the fixed-price results, because only the chokepoint is shocked and the throughput of every other node can only fall.

\paragraph{Stocks.} Initial stocks follow three rules. For 16 countries the stock is the number of days of net-import cover reported by the International Energy Agency times net demand; this gives 183 days for Japan and 191 days for South Korea. The other nodes with positive net demand, among them China, India and the United States, receive a benchmark stock of 90 days of net demand. Results that depend on this benchmark are scenarios and are marked as such. Nodes without net demand hold a working stock of four weeks of throughput. Table~\ref{tab:parameters} collects the rules and constants.

\end{document}